\documentclass[11pt]{article}
\usepackage{graphicx}
\usepackage{amsmath}
\usepackage{amssymb}
\usepackage{amsfonts}
\usepackage{amsthm}
\usepackage{mathtools}
\usepackage{xcolor}
\usepackage{enumitem}
\fontsize{10pt}{14pt}
\usepackage{comment}
\usepackage{hyperref}
\newcommand{\R}{{\mathbb R}}

\newcommand{\weakly}{\rightharpoonup}

\newcommand{\eps}{\varepsilon}

\renewcommand{\ge }{\geqslant}
\renewcommand{\geq }{\geqslant}
\renewcommand{\le }{\leqslant}
\renewcommand{\leq }{\leqslant}
\def\neweq#1{\begin{equation}\label{#1}}
\def\endeq{\end{equation}}

\newtheorem{theorem}{Theorem}[section]
\newtheorem{proposition}[theorem]{Proposition}
\newtheorem{lemma}[theorem]{Lemma}
\newtheorem{corollary}[theorem]{Corollary}

\newtheorem{remark}[theorem]{Remark}
\newtheorem{assumption}[theorem]{Assumption}
\newtheorem{definition}[theorem]{Definition}

\title{Analysis of a nonlinear fish-bone model for suspension bridges \\ with rigid hangers in the presence of flow effects}
\author{Alessio FALOCCHI$^\dag$ - Justin T. WEBSTER$^\ddag$\\
	\footnotesize{$^\dag$ Dipartimento di Matematica, Dipartimento di Eccellenza 2023-2027 - Politecnico di Milano, Italy}\vspace{-3mm} \\\footnotesize{$^\ddag$	University of Maryland, Baltimore County, Baltimore, MD, USA}}
\date{}
\begin{document}
	\maketitle
\begin{abstract}
\noindent We consider a dynamic system of nonlinear partial differential equations modeling the motions of a suspension bridge. This fish-bone model captures the flexural displacements of the bridge deck's mid-line, and each chordal filament's rotation angle from the centerline. These two dynamics are strongly coupled through the effect of  cable-hanger, appearing through a sublinear function.  Additionally, a structural nonlinearity of Woinowsky-Krieger type is included, allowing for large displacements. Well-posedness of weak solutions is shown and long-time dynamics are studied. In particular, to force the dynamics, we invoke a non-conservative potential flow approximation which, although greatly simplified from the full multi-physics fluid-structure interaction, provides a driver for non-trivial end behaviors. We describe the conditions under which the dynamics are uniformly stable, as well as demonstrate the existence of a compact global attractor under all nonlinear and non-conservative effects. To do so, we invoke the  theory of quasi-stability, first explicitly constructing an absorbing ball via stability estimates and, subsequently, demonstrating a stabilizability estimate on trajectory differences applied to the aforesaid absorbing ball. Finally,  numerical simulations are performed to examine the possible end behaviors of the dynamics. 
	
\end{abstract}

{\small
	
}

\section{Introduction}
In this paper we aim to model, analyze, and simulate the dynamics of a suspension bridge deck with cables and hangers, under the effect of some aerodynamical loading. We are particularly interested in the well-posedness and long-time behavior---stability and the existence of a compact global attractor---for the associated non-gradient dynamics \cite{holawe}. In line with previous work on suspension bridges (and/or associated plate models), we will consider nonlinear models with dissipative effects, geometric and structural nonlinearities, as well as non-conservative lower order PDE terms. We are motivated by several previous works which considered suspension bridge models of different types \cite{Berkovits,bongazmor,crfaga,fa1,marchionna,mckenna,Plaut}. The model presented here is termed \textit{fish-bone} in \cite{berchio}, and a linear version of it was mentioned in \cite{yakubovich}; for a full overview on mathematical models for bridges see \cite{Gazz}. 

The central reference for the analytical program here is the recent \cite{bongazlasweb}. In that paper, a simplified nonlinear plate model captures the deck dynamics under the influence of a crude approximation of potential flow across the bridge deck surface. Well-posedness, stability and the existence of attractors are discussed there. In the present work, we utilize a more involved and refined fish-bone model of recent interest for the suspension bridge dynamics. We do so by introducing both cable-hanger nonlinearities, and a large displacement nonlinearity. The former is sublinear, nonlocal and algebraically complex, and the latter is superlinear and nonlocal.
Concerning the modeling of the cable-hanger, different approaches have been presented in the literature; in \cite{fa1,marchionna,mckenna} the cables are assumed to be fixed,  while the hangers can slacken according to some nonlinear law. On the other hand, in \cite{falocchi1}, the cables are movable, while the hangers are inextensible. This last approach is quite common in the engineering literature. We mention the Austrian engineer, Josef Melan \cite{melan}, who in 1888 introduced the so called Melan equation; in particular, he considered the bridge as a combination of a string (the cable) and a beam (the deck) linked through some rigid hangers, assumed to be uniformly distributed along the central span. 
In \cite{crfaga}, a more realistic model for suspension bridges with both cables and extensible hangers is proposed, but this produces a model of high complexity, which, in the corresponding system (of differential inclusions, not equations), is exceptionally difficult to mathematically treat.

In this paper the modeling of cable-hanger nonlinearity is strictly related to the Melan equation, with some adaptations due to the presence of two cables and a fish-bone deck. We also include in the analysis of deck prestressing effects through the Woinowsky-Krieger nonlinearity \cite{woi} and the aerodynamical effects via first order piston theory \cite{ashley,lamorte2,lighthill,Mchugh}---see  details in Subsection \ref{modeling}. 
 We must adapt the aerodynamic approximation utilized in \cite{bongazlasweb} to the variables associated to the fish-bone dynamics, resulting in a new manifestation  of the aerodynamic loading. The combination of these nonlinear terms behave well at the level of individual trajectories, but require careful analysis at the level of trajectory differences, a common feature for large deflection plate and beam models \cite{chla,holawe}. Such analysis is required for uniqueness considerations, as well as the invocation of the so called quasi-stability theory \cite{chueshov,chla} in the long-time behavior analysis, which ultimately yields our main results.

\subsection{Modeling Discussion}\label{modeling}
In this treatment, we consider a model for a suspension bridge with two degrees of freedom, given by the downward (transverse) vertical displacement $w:=w(x,t)$ and the torsional rotation of the barycentric line of the deck $\theta:=\theta(x,t)$; this provides the fish-bone model in Figure \ref{fishbone}. A generic cross section of the bridge is represented in Figure \ref{fishbone} (on the right); the circles are the sections of the two main cables, the vertical elements are the hangers, and the deck is filled in with black. Nonlinearity is introduced into the model through the effects of the cable-hanger and the prestressing of the deck, as well as for the possibility of large deflections. The dynamics are driven by an external force, as well as through several solution-dependent terms which provide a crude approximation of aerodynamic loading. As we are mainly interested in the qualitative behavior of the system, throughout the paper (except for Section \ref{num}) we take the length of the main span equal to $L=\pi$. In Section \ref{num} we provide some results on a physical model using proper $L>0$ and real physical quantities. 

We now state the partial differential equation (PDE) model we aim to analyze, and describe all of the terms involved. Let $T>0$ (including possibly $T=+\infty$), $I:=(0,\pi)$, $I_T:=I\times (0,T)$, then consider 
\begin{equation}
	\left\{\begin{array}{ll}
		w_{tt}+{(\overbrace{\delta+\beta}^{:=\mu})}\, w_t+w_{xxxx}+\big(P-S\int_Iw_{x}^2\big)w_{xx}+\big(f(w,\theta)\big)_x=g-\beta \Upsilon\theta_t-\eta\theta &\text{in }I_T\\
		\frac{\ell^2}{3}\theta_{tt}+\zeta\, \theta_t +\epsilon\theta_{xxxx}-\kappa\theta_{xx}+\big(\overline f(w,\theta)\big)_x=0\,& \text{in }I_T\\
		w=w_{xx}=\theta=\theta_{xx}=0 &\hspace{-20mm}\text{on }\{0,\pi\}\times(0,T)\\
		w(x,0)=w_0(x),\quad\,\,\theta(x,0)=\theta_0(x) &\text{on }\overline I\\
		w_t(x,0)=w_1(x),\quad\theta_t(x,0)=\theta_1(x) &\text{on }\overline I.
	\end{array}\right.
	\label{eq_sist1}
\end{equation}
We denote by $\ell>0$ the semi-width of the deck, we suppose the mass linear density constant and we take  $\delta, \zeta \ge 0$ as structural damping parameters. 
The presence of the coefficient $\ell^2/3$ in the $\theta$ equation \eqref{eq_sist1}$_2$ is reminiscent of the rotational kinetic energy, where the inertial moment is invoked.

\begin{figure}[h]
	\centering
	\includegraphics[width=10cm]{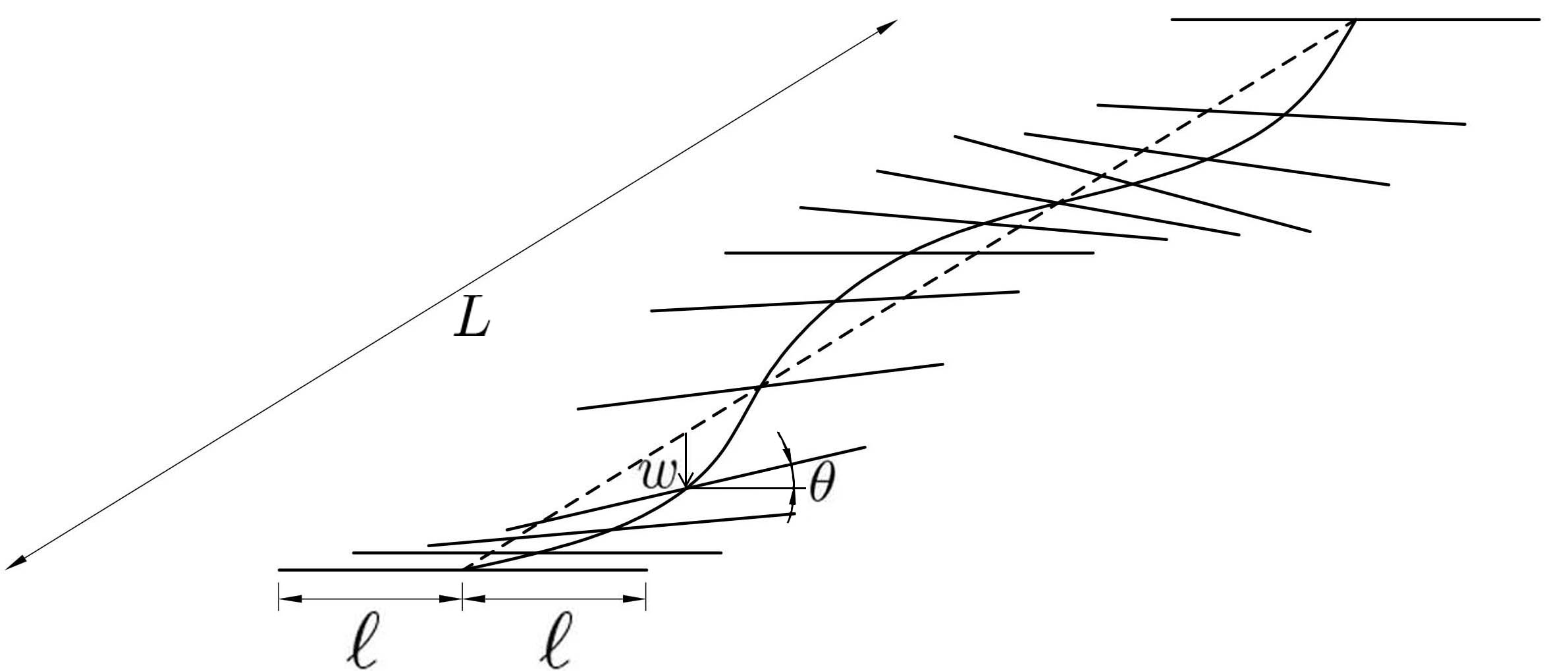}	\includegraphics[width=6cm]{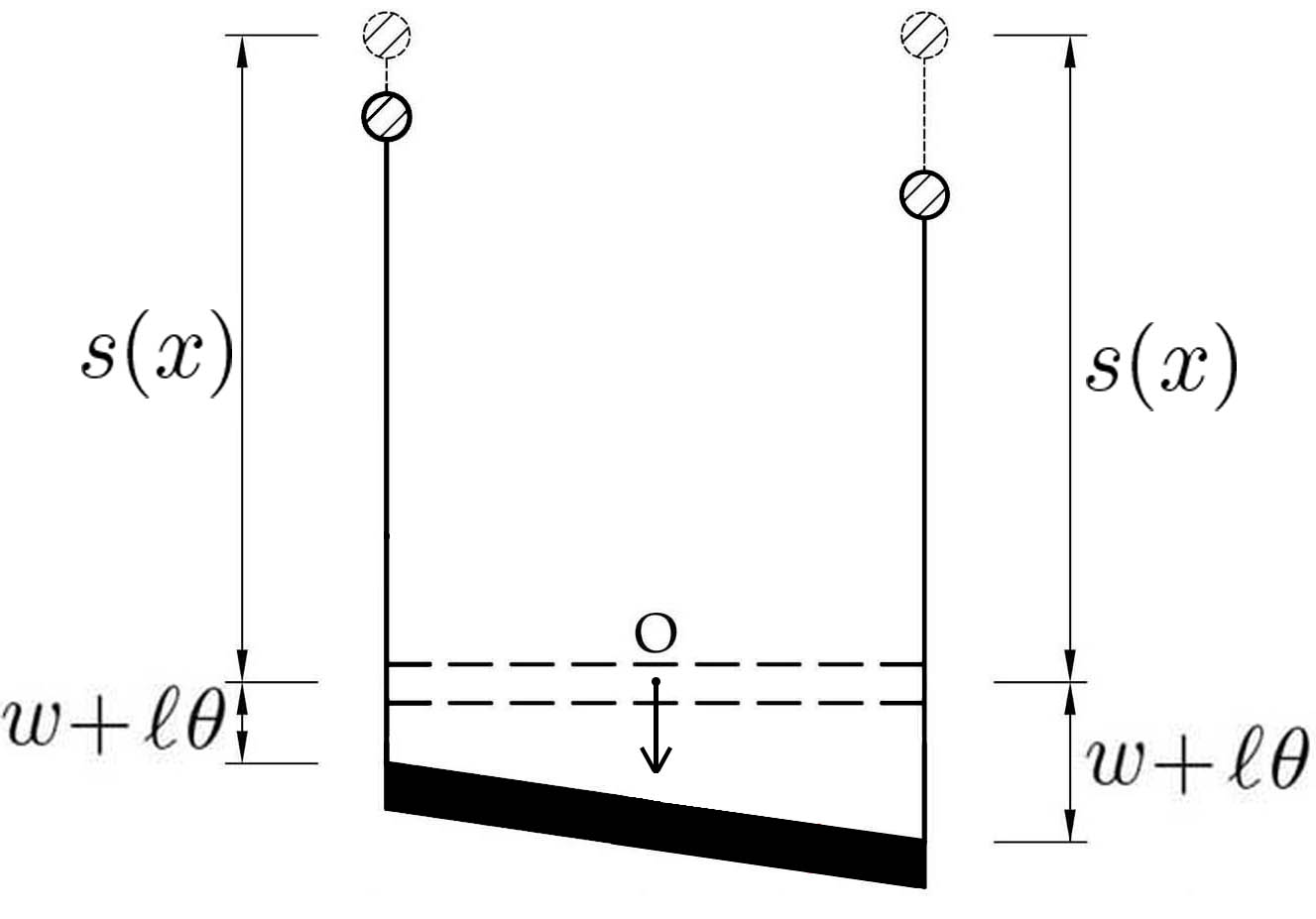}
	\caption{Fish bone model on the left and a cross section for fixed $x$ on the right. Dotted lines are the deck and cables section in rest position, $s(x)$ is the cable initial shape, see Section \ref{cable-nonlinearity}.}
	\label{fishbone}
\end{figure} 
As is common in the engineering literature, we consider linear terms (in space) coming from the bending  and torsional energies of the deck, according to the theory of solids in classical mechanic. We also include, from the Vlasov theory \cite{vlasov}, the linear fourth order term  in $\theta$ related to the warping of the deck, see \eqref{eq_sist1}$_2$. Hence, $\epsilon>0$ is a warping stiffness parameter and $\kappa\geq0$ is the torsional stiffness that we shall further clarify in Section \ref{num}.

To capture the transverse dynamics of the bridge deck, we employ  classical modeling. The dynamcis of Bernoulli and Euler were modified in 1950 by Woinowsky-Krieger \cite{woi} by assuming a nonlinear dependence of the axial strain on the deformation gradient which accounts for stretching due to elongation (``the effect of stretching on bending"). The corresponding term is in the $w$ equation \eqref{eq_sist1}$_1$ and depends on $P\geq0$, a prestressing parameter, and $S>0$, the strength of the nonlinear restoring force resulting from $x$-stretching. Let us note that the nonlocal stretching, giving rise to the superquadratic energy, effects our model (and in general in suspension bridges where $L\gg\ell$) only in the $x$-direction; this is in line with \cite{bongazlasweb}.

Since we are interested in modeling a proper suspension bridge, we add to the consideration of the deck two cables, possessing a parabolic shape $s(x)$ at rest (as in Figure \ref{sketch}). We assume these cables are movable, with rigid hangers, so the resulting nonlinearity is of geometric type. 
\begin{figure}[h]
	\centering
	\includegraphics[width=9.5cm]{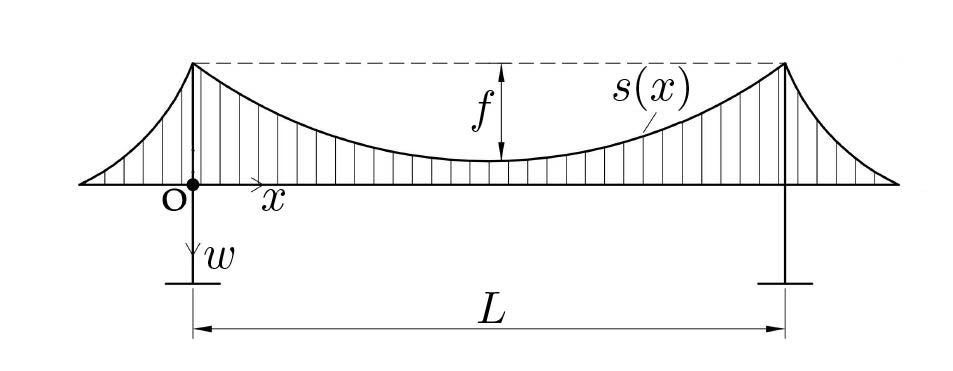}	
	\caption{Sketch of the side view of the suspension bridge with the quotes assumed positive; $s(x)$ is the cable initial shape, see Section \ref{cable-nonlinearity}.}
	\label{sketch}
\end{figure} 
We include the cable nonlinearity through  terms $f(w,\theta)$ and $\overline{f}(w,\theta)$ in \eqref{eq_sist1}. These nonlinearities are benign with respect to the theory of existence, since they are sublinear in their arguments; however, as they are algebraically complex, careful treatment is needed to address their contributions in long-time behavior analysis. In particular, we must control their growth in the Lyapunov-type analysis for the construction of an absorbing set for the dynamics, as well as precise control of the differences of two nonlinear terms for the quasi-stability analysis. We shall specify  $f$ and $\overline f$ precisely in Section \ref{cable-nonlinearity}, as the terms are involved.

As mentioned above, we consider simple aerodynamic loading on the deck of the bridge. The terms we invoke approximate an inviscid, irrotational potential flow across the surface of the bridge deck. The primary focus here is on ability of the wind to destabilize the dynamics and provide non-trivial (i.e.,  non-stationary) end behaviors. In line with \cite{bongazlasweb},  we employ a rudimentary loading, modeled by the so-called piston theoretic approximation of  potential flow. This is a classical theory utilized in the study of aerodynamical systems with high velocities or high frequencies, introduced by Lighthill \cite{lighthill} and developed subsequently by other authors (e.g., \cite{ashley, dowell}). Roughly speaking, it asserts that the pressure acting on an oscillatory slender element, e.g. an airfoil, is comparable to the pressure on an oscillatory 1D piston. Such an approximation discards transport and memory effects, which of course are physically relevant to the real-life dynamics of a suspension bridge. On the other hand, in order to focus on structural nonlinearities and their effects, we invoke this type of loading to incorporate forcing which is not external, but rather based on the interaction of the airflow about the bridge deck and its resulting deformations. This allows us to analyze stability and long-time behavior under some non-conservative aerodynamical loading, without the need for sophisticated moving boundary problems and aeroelastic interactions involving the Navier-Stokes equations (some preliminary such studies include \cite{berchio2,bon2}). Piston-theoretic models can be used to predict the so called flutter instability \cite{bongazlasweb,dowell,Mchugh}, before running expensive computational fluid dynamics (CFD) simulations.

In \cite{lamorte2}, a piston-theoretic approximation is applied to a structure with two degrees of freedom. The difference, with respect to our model here, is that the wind flows in the axial $x$-direction there, while we are interested in so called {\em normal flow}, associated to the chord-wise direction for the bridge deck. We set the problem in the planar domain $\Omega:=(0,\pi)\times (-\ell,\ell)\subset \R^2$, so that, according to the fish-bone model assumptions, we associate a transverse displacement function $u:\overline{\Omega}\times [0,T)\rightarrow\R$ given by 
$$
u(x,\Upsilon,t)=w(x,t)+\Upsilon\tan\theta(x,t)\qquad \forall(x,\Upsilon,t)\in \overline\Omega\times[0,T).
$$ 
Assuming a deck with constant thickness, a flow parallel to the plane containing $\Omega$, small rotations ($\tan\theta\approx\theta$), then the normal velocity of the structure can be represented by
\begin{equation*}\label{v}
	v_N:=\overline \beta \bigg(\!\dfrac{\partial u}{\partial t}+\mathcal{U}\,\dfrac{\partial u}{\partial\Upsilon}\bigg)=	\overline \beta\, (w_t+\Upsilon \theta_t)+\overline\beta\mathcal{U}\theta,
\end{equation*}
where $\overline\beta\in\R_+$ depends on the properties of the fluid flow and $\mathcal{U}$ is the freestream fluid velocity.
Considering the so called first order approximation, we obtain an aeroelastic surface pressure given by
\begin{equation}\label{aero}
\phi_L=-\beta (w_t+\Upsilon\theta_t)-\eta\theta,
\end{equation}
where  $\beta\in\R_+$ and $\eta:=\beta\mathcal{U}\in\R$ include all the  wind parameters.
Introducing $g$ as a vertical stationary loading, e.g.  gravity, and, taking the aerodynamic moment to be zero due to the assumed symmetry, we obtain the system of PDEs in  \eqref{eq_sist1}, complete with  aerodynamic terms; in the sequel we will typically put $\mu:=\delta+\beta$.

Although we make several simplifying assumptions, the model presented here represents a more realistic course of modeling than compared to previous simplified plate models; in Section \ref{preliminaries} we provide some preliminary results in terms of functional tools and existence and uniqueness of solutions to \eqref{eq_sist1}. This class of models 
 are complex mathematically, and take into account nonlinearities as well as geometric coupling effects, see Section \ref{cable-nonlinearity} and Appendix \ref{app2} for further details. Owing to the aspect ratio associated with actual suspension bridges, it is indeed appropriate to consider this class of fish-bone models. With non-conservative flow effects incorporated, we ask questions about boundedness, compactness, and regularity for the bridge dynamics as $t\to\infty$. Additionally, we consider the interplay between the nonlinear effects in the cables and the bridge deck. Through a Lyapunov type analysis, we construct an absorbing ball. Subsequently, we  utilize the powerful quasi-stability \cite{chueshov,chla} here by providing a completely new analysis of this novel model. This results in the construction of a smooth global attractor of finite dimension and additional regularity, and even a so called fractal exponential attractor---see Section \ref{main} for the main results and Appendix \ref{app1} for general results in this context. In Section \ref{num} we study the corresponding linear problem and we perform some numerical simulations on a suspension bridge case of study.  In Section \ref{proofs} we provide  proofs of the main theorems.

%\subsection{Model of interest}

\section{Preliminaries}\label{preliminaries}
In this section we give some preliminaries related to the function spaces needed in order to pose  problem \eqref{eq_sist1} in weak form and address well-posedness of weak solutions. Moreover, we provide some details about the cable nonlinearity and its functional properties; we then compute the energy of the system, highlighting the positive contributions.
\subsection{Function spaces}
We denote by $L^p(I)$ for $1\leq p \leq\infty$ as the usual Lebesgue spaces,  equipped with the norm $\|\cdot\|_{L^p(\Omega)}$, while $W^{k,p}(I)$, for $k\geq0$, are the classical Sobolev spaces.
Throughout the paper we primarily use the Hilbert spaces $L^2(I)$, $H^1_0(I)$ and $(H^2\cap H^1_0)(I)$ endowed with the following scalar product and norms
\begin{equation*}
\begin{split}
	(u,v)_0&=\int_I uv, \qquad (u,v)_{1}=\int_I u_xv_x, \qquad (u,v)_{2}=\int_I u_{xx}v_{xx},\\
\|u\|^2_0&=\int_I |u|^2, \qquad\hspace{1mm} \|u\|^2_{1}=\int_I |u_x|^2, \qquad\hspace{2mm} \|u\|^2_{2}=\int_I |u_{xx}|^2.
	\end{split}
\end{equation*}
The last two norms above are equivalent to the standard norms in $H^1(I)$ and $H^2(I)$, respectively. Indeed, the eigenvalue problems 
\begin{equation*}
\left\{\begin{array}{ll}
v''''=\lambda v &\text{in }I\\
v=v''=0&\text{on }\{0,\pi\}
\end{array}\right.\qquad \left\{\begin{array}{ll}
-v''''=\Lambda v'' &\text{in }I\\
v=v''=0&\text{on }\{0,\pi\}
\end{array}\right.\qquad \left\{\begin{array}{ll}
-v''=\overline\Lambda v &\text{in }I\\
v=0&\text{on }\{0,\pi\}
\end{array}\right.
\label{eig}
\end{equation*}
admit eigenvalues $\sqrt{\lambda}=\Lambda=\overline\Lambda=k^2$ ($k\in \mathbb{Z}\setminus\{0\}$) and the same eigenfunctions $v(x)=a\sin(kx)$, $a\in\R\setminus\{0\}$; if a function $v(x)$ (resp. $v(t)$) depends only on $x$ (resp. $t$) we often use the notation $v'(x)$ (resp. $\dot{v}(t)$) instead of $v_x$ (resp. $v_t$). We write the variational characterization of the first eigenvalue for the previous problems, respectively, as
\begin{equation*}
	\lambda_1=\min\limits_{\varphi\in (H^2\cap H_0^1)(I)}\dfrac{\|\varphi\|^2_{2}}{\|\varphi\|_0^2}\qquad 	\Lambda_1=\min\limits_{\varphi\in (H^2\cap H_0^1)(I)}\dfrac{\|\varphi\|^2_{2}}{\|\varphi\|_{1}^2}\qquad \overline\Lambda_1=\min\limits_{\varphi\in H^1_0(I)}\dfrac{\|\varphi\|^2_{1}}{\|\varphi\|_{0}^2}.
\end{equation*}
This implies, since $\lambda_1=\Lambda_1=\overline\Lambda_1=1$,
\begin{equation}\label{spectral}
	\|\varphi\|_0\leq\|\varphi\|_{2}\qquad\|\varphi\|_0\leq\|\varphi\|_{1}\qquad \|\varphi\|_{1}\leq\|\varphi\|_{2}\qquad \forall \varphi\in (H^2\cap H^1_0)(I).
\end{equation}
We  will also use the standard interpolation 
\begin{equation}\label{interpolation}
	\|\varphi\|_{1}^2\leq C\|\varphi\|_0\|\varphi\|_{2}\qquad \forall \varphi\in (H^2\cap H_0^1)(I),\,\,\, C>0.
\end{equation}
Throughout the paper we will often denote by $C$ a generic positive constant, that may be different from line to line. 

We denote by $\mathcal{H}$ the dual space of $(H^2\cap H^1_0)(I)$ with the corresponding duality pairing $_{-2}\langle \cdot,\cdot\rangle_{2}$.
We will also need the usual duality pairing $_{-1}\langle \cdot,\cdot\rangle_{1}$ corresponding to the space $H^{-1}$, i.e. the dual of $H_0^1(I)$.
We introduce the self-adjoint operators taken with the boundary conditions \eqref{eq_sist1}$_3$:
$A, \overline A:L^2(I)\rightarrow L^2(I)$ given respectively by $Aw=w_{xx}$ and $\overline A\theta=\epsilon \theta_{xx}+\kappa\theta_x$ on the same domain 
$$
\mathcal{D}:=\big\{u\in (H^4\cap H^1_0)(I): u_{xx}=0 \text{ on }\{0,\pi\}\big\};
$$
let us observe that the condition $u_{xx}=0$ on $\{0,\pi\}$ is the natural boundary condition associated both with $w$ and $\theta$ in strong form.

 %According to these norms, and associated energy identity \eqref{energy},
 The natural energy space for the dynamics $y = (w,w_t; \theta,\theta_t)$ is taken to be
\begin{equation*}\label{dynspace}
	Y=\Big((H^2\cap H^1_0)(I)\times L^2(I)\Big)^2,
\end{equation*} 
defined through the norm
\begin{equation*}\label{norm}
\|(w,v; \theta,\varphi)\|^2_Y=\|v\|_0^2+\|w\|_{2}^2+\|\varphi\|^2_{0}+\|\theta\|^2_{2}.
\end{equation*}
\subsection{Precise definition of the cable nonlinearity}\label{cable-nonlinearity}
Let us now define the cable nonlinearities,  denoted by $f(w,\theta)$ and $\overline f(w,\theta)$ in equations \eqref{eq_sist1}. First, we recall the equation representing the position of the two cables at rest
\begin{equation*}\label{cable0}
s(x)=-\frac{a}{2}x^2+\frac{a}{2}\pi x+s_0\qquad \forall x\in \overline I,
\end{equation*}
where $a>0$ is a tension parameter of the cable and $s_0>0$ is the length of the longest hanger---see Figure \ref{sketch}. The choice of this parabolic shape is common practice in the literature, see e.g. \cite{Gazz,von}.
Now we introduce $b,c\geq0$, two parameters related to the mechanical properties of the cable, and the functions
\begin{equation}\label{xii}
\begin{split}
\Xi(u):=&\sqrt{1+[u_x+s_x]^2}\qquad \xi_0:=\sqrt{1+(s_x)^2}\\ \mathcal{L}(u):=&\int_0^\pi \Xi(u)\,dx\qquad\hspace{6.5mm} \mathcal{L}_0:=\int_0^\pi \xi_0\,dx\\
h(u):=&\big[b\big(\mathcal{L}_0-\mathcal{L}(u)\big)-c\,\xi_0\,\big] \frac{u_x+s_x}{\sqrt{1+[u_x+s_x]^2}};
\end{split}
\end{equation}
we observe that $\xi_0=\Xi(0)$ and $\mathcal{L}_0=\mathcal{L}(0)$. We define the cable nonlinearities as 
\begin{equation}\label{melan42}
\begin{split}
f(w,\theta):=&h(w+\ell\theta)+h(w-\ell\theta)\qquad\quad
\overline 	f(w,\theta):=\ell[h(w+\ell\theta)-h(w-\ell\theta)].
\end{split}
\end{equation}
We point out that when $b=c=0$ the cable nonlinearities disappear, giving a fish bone model without cables. To each cable we associate its nonlinear energy functional
\begin{equation}\label{nonlinear}
\begin{split}
\Pi(u)=&\dfrac{b}{2}\big(\mathcal{L}(u)-\mathcal{L}_0\big)^2+c\int_0^\pi \xi_0\big(\Xi(u)-\xi_0\big)\,dx.
\end{split}
\end{equation}
This nonlinearity has some good functional properties, such as sublinear growth, as we outline below in the technical lemmas---see Section \ref{tech_lemma}.

\subsection{Well-posedness}
We now write the problem \eqref{eq_sist1} in weak form, giving the precise definition of weak solution.
\begin{definition}\label{def_weak}
	Let $T>0$,
	$w_0,\theta_0\in H^2\cap H^1_0(I)$ and $g, w_1,\theta_1\in L^2(I)$.
	A weak solution of \eqref{eq_sist1} is a pair $(w,\theta)$ with the regularity
	$$
	w,\theta\in C^0([0,T];(H^2\cap H_0^1)(I))\,\cap\, C^1([0,T], L^2(I))\,\cap \,C^2([0,T], \mathcal{H})
	$$
	such that for all $t\in(0,T]$ and for all $v,\varphi\in (H^2\cap H_0^1)(I)$, we have
	\begin{equation}\label{weak}
		\begin{cases}
			\,_{-2}\langle w_{tt}, v\rangle_{2}+\mu(w_t,v)_0+(w,v)_{2}+\big[S\|w\|^2_1-P\big](w,v)_1-(f(w,\theta),v_x)_0=(g-\beta \Upsilon\theta_t-\eta\theta,v)_0\\
			\frac{\ell^2}{3}	\,_{-2}\langle \theta_{tt}, \varphi\rangle_{2}+\zeta(\theta_t,\varphi)_0+\epsilon(\theta,\varphi)_{2}+\kappa(\theta,\varphi)_{1}-(\overline f(w,\theta),\varphi_x)_0=0.
		\end{cases}
	\end{equation}
\end{definition}
We introduce the energy of the fish-bone model as follows
\begin{equation*}
\begin{split}
	&E(w,\theta):=\dfrac{\|w_t\|_0^2}{2}+\dfrac{\|w\|_{2}^2}{2}+\ell^2\dfrac{\|\theta_t\|^2_{0}}{6}+\epsilon\dfrac{\|\theta\|^2_{2}}{2}+\kappa\dfrac{\|\theta\|^2_{1}}{2}\\
	&\mathcal{E}(w,\theta):=E\big(w,\theta\big)+\Pi\big(w+\ell\theta\big)+\Pi\big(w-\ell\theta\big)+\dfrac{S}{4}\|w\|^4_1-\dfrac{P}{2}\|w\|^2_1-\big(g,w\big)_0,
\end{split}
\end{equation*}
where $\Pi(u)$ is the cable nonlinearity given in \eqref{nonlinear}.
It is useful to isolate the positive contributions of the energy, denoted
\begin{equation*}
\begin{split}
E_+(w,\theta):=&E(w,\theta)+\dfrac{S}{4}\|w\|^4_1+\Pi\big(w+\ell\theta\big)+\Pi(w-\ell\theta).
\end{split}
\end{equation*}
Depending on the context, we may use $E(t)$, $\mathcal{E}(t)$, $E_+(t)$ to emphasize time dependence, or   $E$, $\mathcal{E}$, $E_+$, suppressing all the dependencies.

The next theorem provides weak and strong well-posedness of  \eqref{eq_sist1}, which is proved in Section \ref{proof-teo1}.
\begin{theorem}\label{teo1}
	Assume $T>0$, $g\in L^2(I)$, and let $\mu,\zeta,  \epsilon, \ell,\beta, S>0$, $P,\kappa\geq0$, $\Upsilon\in[-\ell,\ell]$ and $\eta\in \R$. For any initial data $(w_0,w_1;\theta_0,\theta_1)\in Y$ there exists a unique weak solution of \eqref{eq_sist1}. 
		Moreover, if $w_0,\theta_0\in\mathcal{D}$ and $w_1,\theta_1\in (H^2\cap H_0^1)(I)$, then
	$$
	w,\theta\in C^0([0,T];\mathcal{D})\,\cap\, C^1([0,T],  (H^2\cap H_0^1)(I))\,\cap \,C^2([0,T], L^2(I))
	$$ 
	and $u$ is a strong solution of \eqref{eq_sist1}, satisfying the equations and boundary conditions point-wise a.e. $x$ and $t$. 
	
	Any weak solution satisfies, for $0\leq s<t$, the energy identity
	\begin{equation}\label{energy}
		\mathcal{E}(t)+\mu\int_s^t\|w_t(\tau)\|_0^2d\tau+\zeta \int_s^t\|\theta_t(\tau)\|_0^2d\tau=\mathcal{E}(s)-\beta \Upsilon\int_s^t\big(\theta_t(\tau),w_t(\tau)\big)_0d\tau-\eta\int_s^t\big(\theta(\tau),w_t(\tau)\big)_0d\tau.
	\end{equation}
	\end{theorem}
	\begin{remark} Semigroup methods are also amenable to this (and comparable beam and plate) problems, which can provide strong and so called generalized solutions---see \cite{chla}. In all cases, energy estimates and Lyapunov calculations are performed on smooth solutions, and extended by density for weak and/or mild solutions. \end{remark}

In the standard way, the well-posedness (in the weak and strong sense) allow us to define the solution strongly continuous (nonlinear) semigroup $S_t$ on the energy space $Y$. Indeed, we define for any $y_0=(w_0,w_1;\theta_0,\theta_1) \in Y$ $S_t(y_0) = y(t) = (w(t), w_t(t);\theta(t),\theta_t(t))$, which represents the unique weak solution to \eqref{eq_sist1} with initial data $y_0$. As in \cite{chla}, the pair $(S_t,Y)$ represents a {\em dynamical system}, and this is the system to which we will provide our long-time behavior results below in Proposition \ref{prop-absorbing} and our main theorem, Theorem \ref{teo-main}. 
\section{Results and discussion}\label{main}
In stating our results, please  take note of Appendix \ref{app1} for technical definitions concerning dissipative dynamical systems. Additionally, that Appendix includes the abstract theorems we invoke---after proving estimates---which will yield the results stated here. 

\subsection{Statement of main results and context}
Let $(S_t,Y)$ denote the dynamical system associated to the solution semigroup provided by Theorem \ref{teo1}.
In the first proposition we prove that this dynamical system  has a uniform  absorbing ball, under a restriction on the $\theta$-stiffness parameter $\epsilon$. The proof is given in Section \ref{proof-abs}.
\begin{proposition}\label{prop-absorbing}
	Assume $T>0$, $g\in L^2(I)$, with $\mu,\zeta,\ell,\beta, S>0$, $P,\kappa\geq0$,  $\Upsilon\in[-\ell,\ell]$, $\eta\in \R$. Then, there is a ~$\overline\nu=\overline\nu(\mu,\zeta)$, so that for any
	\begin{equation*}\label{assumpt}
		\epsilon\in\bigg(0\,,\,\frac{\ell^2\overline{\nu}^2}{3\beta^2}\bigg)
	\end{equation*}
 the dynamical system $(S_t,Y)$ corresponding to weak solutions to \eqref{eq_sist1} has a uniformly absorbing set $\mathcal{B}_{\bar\nu}$.
\end{proposition}
\begin{remark}
	There are several ways to interpret the numerology of the above result. We  view the result as follows: given the damping parameters $\mu,\zeta>0$, there is a fixed $\overline{\nu}$ which dictates the size of the absorbing set. That absorbing set is uniform for all values of the warping stiffness parameter $\epsilon$ in the above specified range. 
	\end{remark}
 As the dynamics in \eqref{eq_sist1} are non-gradient (with unsigned terms appearing in the energy identity), we do not expect a clean characterization of the global attractor as the unstable manifold of the stationary points \cite{chueshov,chla}. Thus, to show that there is a compact global attractor for these dynamics, it is necessary to construct the absorbing set ``by hand" using Lyapunov methods. Our proof involves a novel Lyapunov function, tailored to the two nonlinearities present here, as well as the non-trivial flow contributions. 

With an absorbing ball for $(S_t,Y)$ in hand, we proceed to our main result  on attractors for the system \eqref{eq_sist1}. This is proved in Section \ref{proof-teo-main}. The proof relies on critically decomposing the nonlinear terms in a novel way; the resulting technical lemmata which support the main result are found in Appendix \ref{app2}. 
\begin{theorem}\label{teo-main}
		Under the same assumptions of Proposition \ref{prop-absorbing}, there exists a compact global attractor $\mathcal{A}$ for the dynamical system $(S_t,Y)$ corresponding to weak solutions to \eqref{eq_sist1} (as in Theorem \ref{teo1}). Moreover,
		\begin{itemize}
			\item $\mathcal A$ is smooth in the sense that $\mathcal{A}\subset (H^4\cap H^1_0)(I)\times  (H^2\cap H^1_0)(I)$ and is a bounded set in that topology;
			\item $\mathcal A$ has  finite fractal dimension in the space $Y$;
			\item there exists a generalized fractal exponential attractor $\widetilde{\mathcal{A}}_{exp}\subset Y$ with finite fractal dimension in $\widetilde{Y}:=(L^2(I)\times \mathcal{H})^2$.
		\end{itemize}
\end{theorem}
The compact global attractor here encapsulates the end behaviors of these flow-driven, fish-bone dynamics. In particular, as we will investigate numerically below, the attractor contains the stationary points for these nonlinear equations as well as the possibility of dynamic end behavior. By invoking  quasi-stability theory \cite{chueshov,chla}, as outlined in Appendix \ref{app1}, we are able to show the first two bullet points above via a single stabilizability estimate on the difference of two trajectories. In order to obtain this estimate, some ``compactness" must be harvested from the nonlinearities, via particular algebraic decompositions of the nonlinear terms (in the aforementioned Section \ref{app2}). Following this, a straightforward estimate on the difference of trajectories in a weaker topology provides the (not necessarily unique) generalized (fractal) exponential attractor. In this scenario, the presence of dissipation in both components of the model is enough to stabilize---in a uniform way---all trajectories to a smooth and finite dimensional set in the state space. By ``fattening" this invariant attractor $\mathcal A$, we are able to observe exponential convergence rates, though the resulting set $\mathcal A_{exp}$ need not  be unique nor finite dimensional in the state space $Y$. In some sense, then,  the structural dissipation present in the model works in conjunction with the large-deflection nonlinearity to 
counter the non-dissipative and non-conservative flow terms scaled by $\Upsilon$ and $\eta$ as well as the nonlinear cable-hanger effects. The result is that the overall dynamics remain bounded, and in fact squeeze down to a ``nice" (and uniform) set.

\subsection{Central challenges and contributions}

We conclude this section by emphasizing some of the novelties and challenges overcome in this contribution:
\begin{itemize}
\item We consider a fully nonlinear structural model, taking into account cable-hanger and nonlinear coupling between the transverse vertical and angular dynamics; additionally, we allow the possibility of large transverse vertical deflections, in line with previous analyses of suspension bridge models \cite{bongazlasweb, crfaga,falocchi1,mckenna}. This nonlinear, coupled hyperbolic-hyperbolic-type structural dynamics has not previously appeared in the literature. 
\item We adapt the rudimentary aerodynamic approximation utilized in previous works to the set of variables relevant here, allowing for the addition of non-conservative, flow-driven loading in the fish-bone dynamics. 
\item In considering these two nonlinear effects and their interaction, it is not obvious that compactness can be extracted  to obtain a stabilizability estimate (more broadly, some notion of asymptotic compactness or asymptotic smoothness \cite{chla}); however, utilizing a new decomposition of the hanger nonlinearity in conjunction with a decomposition for Woinowky-Krieger \cite{holawe} type dynamics, we precisely obtain this asymptotic notion of squeezing/compactness. 
\item To obtain the absorbing ball, low frequencies (lower norms) emanating from the non-dissipative flow terms must be controlled through the structural nonlinearities. While this has been explored in depth for non-conservative beams and plate models \cite{chla,holawe}, it was novel to consider whether the Woinowsky-Krieger nonlinearity---acting only in the transverse vertical variable---is sufficient to provide this control for the {\em coupled} dynamics. Indeed, a central challenge is to get estimates for the Lyapunov function, including terms in $\theta$, using the control provided Lemma \ref{eps-lemma} appearing only in the transverse $w$ variable. Careful attention and sharpness is required in the process of obtaining the estimates in Seciton \ref{proof-abs}. Additionally, a novel, adapted Lyapunov function in Section \ref{proof-abs} accommodates the nonlinear coupling, as well as the lower order flow terms in $\theta$. 
\item As demonstrated in \cite{bongazlasweb,holawe}, non-conservative bridge models with Woinowksy-Krieger (or Berger plate) nonlinearities are good candidates for the direct application of quasi-stability, including harvesting its most powerful theorems in application. Indeed, by constructing the absorbing ball $\mathcal B$ in Proposition \ref{prop-absorbing} and obtaining a good decomposition of the nonlinear difference dynamics on $\mathcal B$, we can invoke Corollary \ref{doy*} to obtain the existence of the attractor, as well as its smoothness and finite dimensionality in ``one shot". 
\item Finally, we remark that there is a clear dependence of our results on the strength of the coupling through the $\theta$-stiffness parameter, $\epsilon$. At present, we do not feel like this can be eliminated without requiring large damping through the $\zeta$-parameter. 
\end{itemize}

\begin{remark}[Damping and linearity in the model]
The linear model (taking both nonlinear effects null) results in a one-way or partially coupled model. For this model, one can explicitly solve for $\theta$ using a series, and plug the result into the RHS of the $w$-dynamics. In Section \ref{num} we discuss this further, including exponential stability and the possibility of vanishing damping coefficients. A particularly salient case is the one where no damping is imposed in the system at all, other than what comes from the aerodynamics terms in \eqref{aero}. For our results on attractors for the nonlinear model, we note that the presence of the nonlinear cable-hanger provides coupling in the system, so that it is in fact (nonlinearly) strongly coupled. In this case, if it is of course to be expected that dissipation is required in both components of the model to expect stability to a nice set in the state space $Y$. 
\end{remark}

\section{Explicit solutions and numerical results}\label{num}
In this section we present some numerical experiments on the system \eqref{eq_sist1}, taking $g$ constant and $I=(0,L)$,  $L$ being the length of the span of the bridge; in the spirit of the proof of Theorem \ref{teo1}, we apply the Galerkin procedure, approximating by in vacuo structural eigenfuctions (modes). This modal approach has been utilized recently in the mathematically-themed numerical works \cite{holawe,HHWW}. Since the eigenfunctions of standard elasticity operators form a basis for the state space, a good well-posedness
result for the full system justifies this kind of approximation. Hence,
we reduce the evolutionary PDE to a finite dimensional system of ODEs via modal truncation.

More precisely, given the boundary conditions, we seek approximated solutions in the form
\begin{equation}
	w(x,t)=\sqrt{\frac{2}{L}}\sum_{j=1}^{N} w_j(t) \hspace{1mm}\sin\bigg(\dfrac{j\pi x}{L}\bigg), \hspace{5mm} \theta(x,t)=\sqrt{\frac{2}{L}}\sum_{j=1}^{N} \theta_j(t) \hspace{1mm}\sin\bigg(\dfrac{j\pi x}{L}\bigg).
	\label{approx}
\end{equation}
with $N\geq 1$ and initial data
\begin{equation}
	w_i(x)=\sqrt{\frac{2}{L}}\sum_{j=1}^{N} w_j^i \hspace{1mm}\sin\bigg(\dfrac{j\pi x}{L}\bigg), \hspace{5mm} \theta_i(x)=\sqrt{\frac{2}{L}}\sum_{j=1}^{N} \theta_j^i \hspace{1mm}\sin\bigg(\dfrac{j\pi x}{L}\bigg)\qquad (i=0,1).
	\label{ics}
\end{equation}

In the next subsection we give a preliminary analysis on the corresponding linear problem, where the shape of the eigenfunctions (identical for both $w$ and $\theta$) allows to find the exact representation of the solution, i.e. $N\rightarrow\infty$ in \eqref{approx}.

\subsection{Linear modal and stability analysis}
In this section we consider the linear problem, i.e. \eqref{eq_sist1} with $P=S=0$ and $f=\overline f\equiv0$, or equivalently $b=c=0$ in \eqref{xii}--\eqref{melan42}; hence, we obtain
\begin{equation}
	\left\{\begin{array}{ll}
		w_{tt}+\mu\, w_t+w_{xxxx}=g-\beta \Upsilon\theta_t-\eta\theta &\text{in }I_T\\
		\frac{\ell^2}{3}\theta_{tt}+\zeta\, \theta_t+\epsilon\theta_{xxxx} -\kappa\theta_{xx}=0\,& \text{in }I_T\\
		w=w_{xx}=\theta=\theta_{xx}=0 &\text{on }\{0,L\}\times(0,T)\\
		w(x,0)=w_0(x),\quad\,\,\theta(x,0)=\theta_0(x) &\text{on }\overline I\\
		w_t(x,0)=w_1(x),\quad\theta_t(x,0)=\theta_1(x) &\text{on }\overline I.
	\end{array}\right.
	\label{eq_lin}
\end{equation}
We note immediately that this linear problem is not strongly coupled. Indeed, the presence of $f,\overline f$ provide nonlinear coupling (through the cable-hanger) in \eqref{eq_sist1}. With the one-way coupled linear problem in \eqref{eq_lin}, we can directly utilize the series expansions introduced in \eqref{approx}. Moreover, we benefit from the unique structure of the fish-bone equations, providing the same mode functions for both solution variables in the system. 
Plugging \eqref{approx} into \eqref{eq_lin}, testing by $\sqrt{2/L}\sin(\pi kx/L)$ with $k=1,2,\dots$ and integrating over $(0,L)$ we produce the following system of ODEs
\begin{equation}
	\begin{cases}
		\ddot{w}_j(t)+\mu\dot{w}_j(t)+\frac{j^4\pi^4}{L^4} w_j(t)=-\beta\Upsilon\dot{\theta}_j(t)-\eta \theta_j(t)+g\dfrac{\sqrt{2L}(1-(-1)^j)}{j\pi}\\
		\dfrac{\ell^2}{3}\ddot{\theta}_j(t)+\zeta\dot{\theta}_j(t)+\big(\epsilon\frac{j^4\pi^4}{L^4}+\kappa\frac{j^2\pi^2}{L^2}\big)\theta_j(t)=0\\
		w_j(0)=w_j^0,\quad \dot{w}_j(0)=w_j^1\\
		\theta_j(0)=\theta_j^0,\quad\,\,\, \dot{\theta}_j(0)=\theta_j^1
	\end{cases}\quad j=1,2,\dots.
	\label{odes}
\end{equation}
Noting the decoupling, we can explicitly plug the $\theta_j$ solutions into the equation for $w_j$. 
Rewriting the system in matrix form, using reduction of order, for each $j$ fixed we obtain a first order linear system of four equations. That system has eigenvalues $\lambda$ which must satisfy the characteristic equation
\begin{equation}\label{eig0}
\bigg(\dfrac{\ell^2}{3}\lambda^2+\zeta\lambda+\epsilon\frac{j^4\pi^4}{L^4}+\kappa\frac{j^2\pi^2}{L^2}\bigg)\bigg(\lambda^2+\mu\lambda+\frac{j^4\pi^4}{L^4}\bigg)=0.
\end{equation} 
When we assume the damping coefficients are positive, so $\mu,\zeta>0$, we see that the real part of all eigenvalues is  negative. Thus we observe, immediately, that solutions to \eqref{eq_lin} are always exponentially stable. We state this as a lemma. 
\begin{lemma}\label{lemma00}
		Let $ \mu,\zeta,\epsilon,\ell, L,\beta>0$, $\kappa\geq 0$, $\Upsilon\in[-\ell,\ell]$, $\eta\in\mathbb{R}$ and $g=0$. Let $(w(t),w_t(t);\theta(t),\theta_t(t))$ be the corresponding weak solution to \eqref{eq_lin}. Then
		there exists $\gamma,M>0$ such that 
		$$
		\|(w(t),w_t(t);\theta(t),\theta_t(t))\|_Y\leq Me^{-\gamma t}\|(w_0,w_1;\theta_0,\theta_1)\|_Y.
		$$
\end{lemma}

\begin{remark} The behavior of the decoupled linear problem is, in some sense, elementary, since it is clear that we have exponential stability. In particular, eschewing the series solution, we see that the damping present in the $\theta$ equation is sufficient to provide exponential stability for the $\theta$ dynamics. At this point, one can invoke the variation of parameters formula for the $w$ dynamics taking $\theta$ terms as give on the RHS. With damping present there and exponential decay of $\theta$, it is clear that exponential decay follows immediately for $w$. We point out that the behavior of the linear system is not a viable predictor of long time behavior for the nonlinear system of interest, owing to the fact that the cable-hanger nonlinearity itself introduces the nonlinear coupling into the problem. For this reason, the behavior of the nonlinear dynamics in \eqref{eq_sist1} remain highly non-trivial and of interest here; we demonstrate the existence of a global attractor for those coupled dynamics, in addition to investigating qualitative features of the dynamics via numerical experiments in this section.
\end{remark}

In this context for large damping, overdamped solutions may appear, which are in some sense not realistic. Hence, we focus on the realistic case of small damping, which  we describe in the next proposition, proved in Section \ref{proof-lin}.
\begin{proposition}\label{lin}
Let $\epsilon,\ell, L,\beta>0$, $\kappa\geq 0$, $\Upsilon\in[-\ell,\ell]$, $g,\eta\in\mathbb{R}$. Moreover, assume $\zeta\neq\frac{\ell^2}{3}\mu$, $0<\zeta<\frac{2\pi\ell}{L\sqrt{3}}\sqrt{\epsilon\frac{\pi^2}{L^2}+\kappa}$ and $0<\mu<\frac{2\pi^2}{L^2}$, then the solution of \eqref{eq_lin} admits the following representation
\begin{equation}\label{solution}
	\begin{split}
		&w(x,t)=\sqrt{\frac{2}{L}}\sum_{j=1}^{\infty} w_j(t) \hspace{1mm}\sin\bigg(\dfrac{j\pi x}{L}\bigg), \qquad \theta(x,t)=\sqrt{\frac{2}{L}}\sum_{j=1}^{\infty} \theta_j(t) \hspace{1mm}\sin\bigg(\dfrac{j\pi x}{L}\bigg),
	\end{split}
\end{equation}
where 
\begin{equation}\label{theta}
	\begin{split}
	w_j(t)=&e^{-\frac{\mu}{2}t}\bigg[c^j_1\sin\bigg(\frac{\omega_j}{2}t\bigg)+c^j_2\cos\bigg(\frac{\omega_j}{2} t\bigg)\bigg]+g\dfrac{L^4\sqrt{2L}(1-(-1)^j)}{j^5\pi^5}\\+&e^{-\frac{3\zeta}{2\ell^2}t}\bigg[A_j\sin\bigg(\frac{3}{2\ell^2}\gamma_j t\bigg)+B_j\cos\bigg(\frac{3}{2\ell^2}\gamma_j t\bigg)\bigg]\hspace{28mm} j\in\mathbb{N}_+,\\
	\theta_j(t)=&e^{-\frac{3\zeta}{2\ell^2}t}\bigg[\dfrac{1}{\gamma_j}\bigg(\dfrac{2\ell^2}{3}\theta^1_j+\zeta\theta_j^0\bigg)\sin\bigg(\frac{3}{2\ell^2}\gamma_j t\bigg)+\theta_j^0\cos\bigg(\frac{3}{2\ell^2}\gamma_j t\bigg)\bigg]\quad j\in\mathbb{N}_+,
	\end{split}
\end{equation}
$\omega_j:=\sqrt{\frac{4j^4\pi^4}{L^4}-\mu^2}$, $\gamma_j:=\sqrt{\frac{4j^2\pi^2\ell^2}{3L^2}\big(\epsilon\frac{j^2\pi^2}{L^2}+\kappa\big)-\zeta^2}$ and the coefficients $A_j,B_j,c_1^j,c_2^j \in\mathbb{R}$ are  computed in \eqref{A-B}--\eqref{c1-c2}.
The series converges uniformly in $I$ for all $t\geq0$.
\end{proposition}
Let us observe that for $\zeta=\frac{\ell^2}{3}\mu$ it is possible to have a different solution form. This occurs if also $\omega_j=\tfrac{3}{\ell^2}\gamma_j$; in this case, the solution $w$ slightly changes, i.e. the second line  in \eqref{theta} is multiplied by $t$. As indicated by Lemma \ref{lemma00}, nothing changes qualitatively, as exponential stability can still observed.

While the partially coupled system in \eqref{eq_lin} (and modally in \eqref{odes}) is straightforward to examine when both damping parameters are positive, $\mu,\zeta>0$, an interesting question arises when damping effects are omitted in the $\theta$ dynamics. In particular, in \cite{TNB} some experiments  are reported  on a model similar to the Tacoma Narrows bridge, showing that
the damping values are of the order of 1\%. We may assume that such value concerns only the vertical displacements, i.e. $\delta$, while for the torsional damping $\zeta$ we have no viable information. One may thus assume $\delta=\zeta$ as in \cite{mckenna,Plaut}, or more conservatively,  one can put $\zeta=0$. In the latter case from a mathematical view point, however, the flow provides some dissipative effects. This has been noted in several places, when potential flow interacts with elastic deformations \cite{chla,thinair,holawe,attractors}. In the context of the present flow description in \eqref{aero}, we similarly observe the contribution of dissipation. In particular, we see that we can take both intrinsic, {\em structural} damping parameters to be zero in the dynamics---$\delta=\zeta=0$. From the flow, then, we obtain $\theta$ damping in the $w$ equation: 
$$w_{tt}+\beta \, w_t+w_{xxxx}=g-\beta \Upsilon\theta_t-\eta\theta ~~~\text{in }I_T.$$ An interesting question, therefore, is the extent to which this dissipation can be harvested for the $\theta$ dynamics, through the $w$ dynamics, when there is no imposed structural damping. The next lemma answers this question in the negative, which is contrary to the outcome in \cite{attractors}. We omit the proof since it follows as for Proposition \ref{lin}. 

\begin{lemma}
		Let $ \delta=\zeta=0$, $\epsilon,\ell, L,\beta>0$, $\kappa\geq 0$, $\Upsilon\in[-\ell,\ell]$, $g, \eta\in\mathbb{R}$. Let $w_j(t)$ and $\theta_j(t)$ for $j \in \mathbb N_+$ correspond to the solutions to \eqref{odes}. Then the Fourier modes obey:
		\begin{equation*}\label{theta2}
			\begin{split}
				w_j(t)=&e^{-\frac{\beta}{2}t}\bigg[c^j_1\sin\bigg(\frac{\omega_j}{2}t\bigg)+c^j_2\cos\bigg(\frac{\omega_j}{2} t\bigg)\bigg]+g\dfrac{L^4\sqrt{2L}(1-(-1)^j)}{j^5\pi^5}\\+&A_j\sin(\gamma_jt)+B_j\cos(\gamma_jt)\hspace{3mm} j\in\mathbb{N}_+,\\
				\theta_j(t)=&\dfrac{\theta^1_j}{\gamma_j}\sin(\gamma_j t)+\theta_j^0\cos(\gamma_jt)\quad j\in\mathbb{N}_+.
			\end{split}
		\end{equation*}
	$\omega_j:=\sqrt{\frac{4j^4\pi^4}{L^4}-\mu^2}$, $\gamma_j:=\frac{\sqrt{3}j\pi}{\ell L}\sqrt{\epsilon\frac{j^2\pi^2}{L^2}+\kappa}$ and some coefficients $A_j,B_j,c_1^j,c_2^j \in\mathbb{R}$. The series converges uniformly in $I$ for all $t\geq0$.
\end{lemma}

Let us note that in the case when no damping is imposed in the $\theta$ dynamics, so $\zeta=0$ in \eqref{eq_lin}, the eigenvalues corresponding to $\theta_j(t)$ solutions  are pure imaginary 
$$\lambda=\pm i\frac{\sqrt{3}j\pi}{\ell L} \sqrt{\epsilon\frac{j^2\pi^2}{L^2}+\kappa}.$$ 
It is clear, then that the resulting system is Lyapunov stable, but not asymptotically stable, owing to the persistence of $\theta_j$ eigenfunctions which need not decay in time. We also observe, however, that resonance is not possible---which is to say that the damping in the $w_j$ modes is enough to offset the RHS contributions from $g$ and from the flow terms. 

\subsection{Nonlinear modal analysis: a case of study}\label{nonl}
In this section we rewrite \eqref{eq_sist1}, considering the mechanical parameters involved in a real suspension bridge structure.
More precisely, we introduce the (constant) mass linear density of the deck $M$ and we study
\begin{equation*}
	\left\{\begin{array}{ll}
		Mw_{tt}+(\delta+\beta) w_t+D w_{xxxx}+\big(P-S\int_0^Lw_{x}^2\big)w_{xx}+\big(f(w,\theta)\big)_x=Mg-\beta\Upsilon\theta_t-\eta\theta &\text{in } I_T\\
		\frac{M\ell^2}{3}\theta_{tt}+\zeta\, \theta_t+\epsilon\theta_{xxxx}-\kappa\theta_{xx} +\big(\overline{ f}(w,\theta)\big)_x=0\,&\text{in }  I_T\\
		w=w_{xx}=\theta=\theta_{xx}=0 &\hspace{-20mm}\text{on }\{0,L\}\times(0,T)\\
		w(x,0)=w_0(x),\quad\,\,\theta(x,0)=\theta_0(x) &\text{on }\overline I\\
		w_t(x,0)=w_1(x),\quad\theta_t(x,0)=\theta_1(x) &\text{on }\overline I,
	\end{array}\right.
	\label{eq_sist2}
\end{equation*}
where $D=E\mathcal{I}$, $\epsilon=EJ$, $\kappa=GK$, being $E$ the Young modulus, $\mathcal{I}$ the moment of inertia, $G$ the shear constant, $K$ the torsional constant and $J$ the warping constant of the deck. With an abuse of notation, we take, $g$($=9.8\frac{m}{s^2}$) as the  acceleration due to gravity, so that the constant load $Mg$ is the weight of the deck.
Concerning the cable nonlinearity, see Section \ref{cable-nonlinearity}; we take $a=\frac{Mg}{2H}$, $b=\frac{A_cE_c}{\mathcal{L}_0}$ and $c=H$ where $H$ is the cable horizontal tension, $E_c$ is the Young modulus, $A_c$ the sectional area and $\mathcal{L}_0$ the length of the two cables. In the Woinowsky-Krieger nonlinearity, $P$ represents the deck tension in the rest position and $S=\frac{AE}{2L}$ with $A$ the deck cross section area, see e.g., \cite{narciso}.

From the piston theoretic flow approximation \cite{lamorte2,Mchugh} we have considered %$\beta=\frac{2\gamma}{\sigma}$,where $\gamma$ is the specific heat ratio of the fluid, $\sigma$ is the sound velocity in the freastream fluid  
$\eta=\beta\mathcal{U}$ where $\mathcal{U}$ is the freestream fluid velocity and $\beta\in\mathbb{R}_+$ depends on the properties of the fluid. In line with \cite{HHWW,vedeneev} we take $\beta$ varying in the range $10^{-5}\leq \beta\leq 10^{-2}$ (without specifying the unit of measure); about the wind velocity we consider the range $|\mathcal{U}|\leq 30m/s$.

We fix the structural parameters, considering the case of the Tacoma Narrows bridge (TNB), for which there is a vast literature; the mechanical parameters are summarized in Table \ref{tab mec}, see \cite{TNB,crfaga,fa1,Plaut}.

\begin{table}[h]
	\begin{center}
		\begin{tabular}{*{3}{c}}
			\hline
			$E$&$210\hspace{1mm}000MPa$& Young modulus of the deck (steel)\vspace{-3mm}\\\vspace{-3mm}
			$E_c$&$185\hspace{1mm}000MPa$& Young modulus of the cables (steel)\\\vspace{-3mm}
			$G$&$81\hspace{1mm}000MPa$& Shear modulus of the deck (steel)\\\vspace{-3mm}
			$L$&$853.44m$& Length of the main span\\\vspace{-3mm}
			$\ell$&$6m$& Half width of the deck\\\vspace{-3mm}
			$f$&$70.71m$& Sag of the cable, see Fig. \ref{sketch}\\\vspace{-3mm}
			$\mathcal{I}$&$0.154m^4$& Moment of inertia of the deck cross section\\\vspace{-3mm}
			$K$&$6.07\cdot10^{-6}m^4$& Torsional constant of the deck\\\vspace{-3mm}
			$J$&$5.44m^6$& Warping constant of the deck\\\vspace{-3mm}
			$A$&$\approx 1.85m^2$& Area of the deck cross section\\\vspace{-3mm}
			$A_c$&$0.1228m^2$& Area of the cables section\\\vspace{-3mm}
			$M$&$7198kg/m$&Mass linear density of the deck \\\vspace{-3mm}
			$H$&$45\hspace{1mm}413kN$& Horizontal tension in the cables, $H=\frac{MgL^2}{16f}$\\\vspace{-0mm}
			$\mathcal{L}_0$&$868.815m$& Initial length of the cables, see \eqref{xii}\\
			\hline
		\end{tabular}
	\end{center}
\vspace{-4mm}
	\caption{TNB mechanical features.}
	\label{tab mec}
\end{table}
For the modal approximation \eqref{approx}, let us recall some important facts about TNB.
When we speak about a mode like $\sin(\frac{k\pi}{L}x)$, we refer to a motion with $k-1$ nodes, in which the latter are the zeros of the sine function in $(0, L)$. Some meaningful witnesses during the Tacoma Narrows bridge collapse led our modeling choices. From \cite[p.28]{TNB} we know that \textit{“seven different motions have been definitely identified on the main span of the bridge”}. The morning of the failure, Prof. Farquharson described a torsional motion like $\sin(\frac{2\pi}{L}x)$, writing \cite[V-2]{TNB} \textit{“a violent change in the motion was noted. \dots the motions, which a moment before had involved a number of waves (nine or ten) had shifted almost instantly to two \dots the node was at the center of the main span and the structure was subjected
	to a violent torsional action about this point”}.
From \cite{TNB} we learn that in the TNB case, oscillations with
more than 10 nodes on the main span were never seen. Hence, we consider the first 10 transverse vertical modes
and the first 4 torsional modes; this is a good compromise between limiting computational burden and
focusing on the real phenomena viewed in the TNB disaster.
\begin{definition}
	We call $\overline{w}_j(t):=\sqrt{\frac{2}{L}}w_j(t)$ the \textbf{jth transverse vertical mode} and $\overline{\theta}_j(t):=\sqrt{\frac{2}{L}}\theta_j(t)$ the \textbf{jth torsional mode}.
\end{definition}
Consequently, $\overline{w}_j^0:=\sqrt{\frac{2}{L}}w_j^0$, $\overline{w}_j^1:=\sqrt{\frac{2}{L}}w_j^1$ are respectively the initial amplitude and velocity of transverse vertical oscillation, similarly for the $\theta$ initial conditions in \eqref{ics}.
According to the observations during the TNB collapse we fix the initial condition exciting the 9th transverse vertical mode, i.e. $\overline w_9^0=3m$, and we apply an initial condition 10$^{-3}$ smaller on all the other components, i.e.
\begin{equation*}
	\begin{split}
		&\overline{w}_{j}^0=10^{-3}\cdot \overline{w}_{9}^0,\hspace{3mm}\forall j\neq 9,\qquad\overline{\theta}_{j}^0=\overline{w}_{j}^1=\overline{\theta}_{j}^1=10^{-3}\cdot \overline{w}_{9}^0,\hspace{3mm}\forall j.
	\end{split}
	\label{ic22}
\end{equation*}
It is not the purpose of this analysis to show how the results are affected by the choice of the excited mode at the initial time, as many tests in this direction (without the wind) are performed in \cite{crfaga,fa1}. Here we are interested in observing how the torsional modes,  in particular the 2nd, behave  with respect to the presence (or lack thereof) of the dampings, the Woinowsky-Krieger nonlinearity, and the flow parameters. 

Let us consider at first the system undamped, i.e. $\delta=\zeta=0$, without prestressing of the deck as in the real TNB, i.e. $P=S=0$, and no wind in the system, i.e. $\beta=0$. 
In Figure \ref{num1} we plot the $\overline\theta_j(t)$ coefficients with $j=1,\dots,4$, observing that they continue to oscillate around their initial datum. For brevity we do not plot the transverse vertical modes, since for the bridges the most dangerous are the torsional ones, see \cite{Gazz}. 
\begin{figure}[h]
	\centering
	\includegraphics[width=17cm]{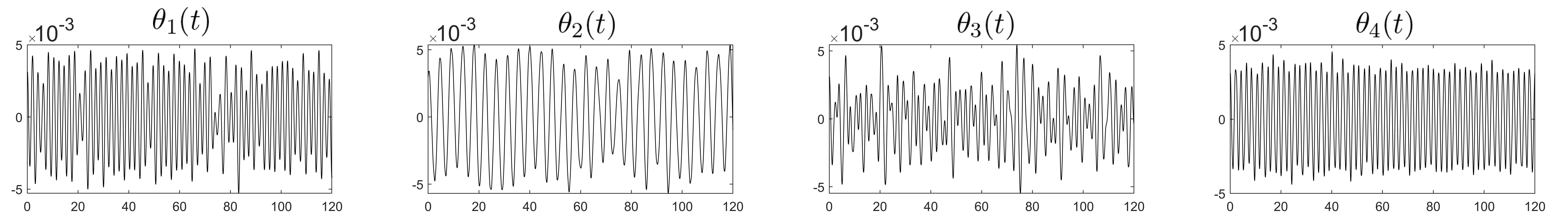}
	\vspace{-5mm}
	\caption{Plots $\overline \theta_j(t)$ ($j=1,\dots,4$) on $[0,120s]$ with $\delta=\zeta=0$, $P=S=0$ and $\beta=0$.}
	\label{num1}
\end{figure} 

On the other hand in Figure \ref{num2} we plot the $\overline\theta_j(t)$ coefficients with $j=1,\dots,4$ in the extremal case $\beta=10^{-2}$ where the wind is very strong $\mathcal{U}=30m/s$.
\begin{figure}[h]
	\centering
	\includegraphics[width=17cm]{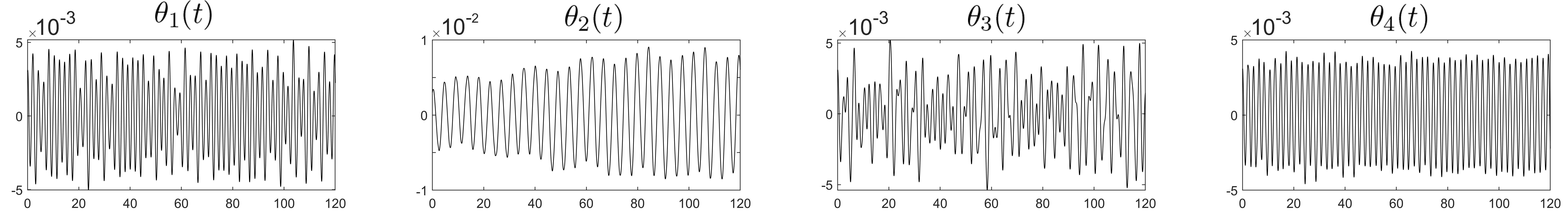}
	\vspace{-5mm}
	\caption{Plots $\overline\theta_j(t)$ ($j=1,\dots,4$) on $[0,120s]$ with $\delta=\zeta=0$, $P=S=0$ and $\beta=10^{-2}$, $\mathcal U=30m/s$.}
	\label{num2}
\end{figure} 
In this case the 2nd torsional mode sees amplitude increase, predicting a possible uncontrolled growth beyond the time-lapse simulation; the other three torsional modes oscillate more or less around their initial datum.  Here and in all the simulations in this section, we put $\Upsilon=\ell$ observing that the results are not  meaningly affected if we take $\Upsilon=0,-\ell$ or $\mathcal U$ with opposite sign; indeed, the model is essentially symmetric with respect to the sign of $\mathcal{U}$. 

If we introduce into the system the nonlinear contribution due to the effect of stretching on bending, i.e. $P=0$ and $S>0$, in absence of wind we obtain a behavior which is qualitatively similar to that reported in Figure \ref{num1}; this is consistent with \cite{HHWW,holawe}. Including the effect of the wind, we observe that the Woinowsky-Krieger nonlinearity acts, as expected, in favor of stability, see Figure \ref{num3}; there the $\theta_2(t)$ growth is less marked than in Figure \ref{num2}. This stabilizing behavior is also observed for a beam model in \cite{HHWW}, where the ``strength" of the nonlinearity prohibits arbitrary growth of the elastic displacements.

\begin{figure}[h]
	\centering
	\includegraphics[width=17cm]{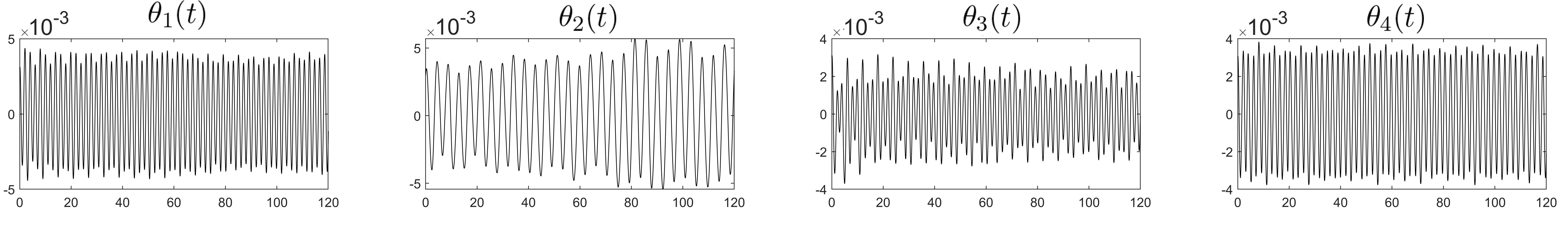}
	\vspace{-4mm}
	\caption{Plots $\overline\theta_j(t)$ ($j=1,\dots,4$) on $[0,120s]$ with $\delta=\zeta=0$, $P=0$, $S=\frac{EA}{2L}$ and $\beta=10^{-2}$, $\mathcal U=30m/s$.}
	\label{num3}
\end{figure} 

If we include in the system the damping effects in both model components, i.e. $\delta=\zeta=0.01$ \cite{TNB,mckenna}, we obtain the expected damped oscillations for torsional modes, although the wind parameters play an important role, see Figure \ref{num4}.
\begin{figure}[h]
	\centering
	\includegraphics[width=17cm]{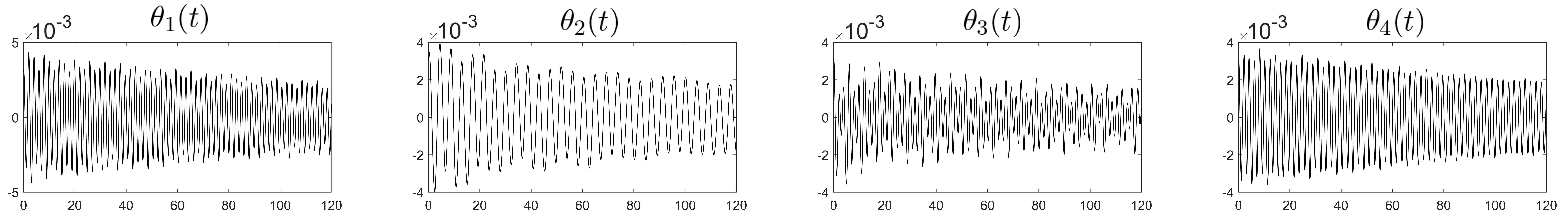}
	\vspace{-3mm}
	\caption{Plots $\overline\theta_j(t)$ ($j=1,\dots,4$) on $[0,120s]$ with $\delta=\zeta=0.01$, $P=0$, $S=\frac{EA}{2L}$ and $\beta=10^{-2}$, $\mathcal U=30m/s$.}
	\label{num4}
\end{figure} 
If we damp only the $w$ component, i.e.  $\delta=0.01$ and $\zeta=0$, we qualitatively obtain torsional oscillations around the initial datum, as in the unforced (and undamped) situation, see Figure \ref{num1}.

The plots showed in this section are representative of an extremal wind condition $\beta=10^{-2}$ and $\mathcal{U}=30m/s$; for lower values of $\beta$ and $\mathcal{U}$ in the ranges declared at the beginning of this section we obtain intermediate situations which qualitatively fall between those  presented here in presence or absence of wind. 
%\begin{figure}[h]
%	\centering
%	\includegraphics[width=17cm]{"num5.jpg"}
%	\caption{Plots $\overline\theta_j(t)$ ($j=1,\dots,4$) on $[0,120s]$ with $\delta=0.01$, $\zeta=0$, $P=0$, $S=\frac{EA}{2L}$ and $\beta=10^{-2}$, $\mathcal U=30m/s$.}
%	\label{num5}
%\end{figure} 

\section{Proofs of stated theorems}\label{proofs}

We begin with two technical lemmas concerning the nonlinear structure of the cable nonlinearity. We then move on to the proof of Theorem \ref{teo1} concerning weak and strong well-posedness. Subsequently, we construct an absorbing ball for the dynamical system $(S_t,Y)$ associated with weak solutions in Proposition \ref{prop-absorbing}; this is done through Lyapunov methods. Finally, we show that the dynamical system $(S_t,Y)$ is {\em quasi-stable} on the absorbing ball, which yields Theorem \ref{teo-main}. 
\subsection{Technical lemmas on the cable nonlinearity}\label{tech_lemma}
We give here some technical lemmas on the cable nonlinearity that we use throughout the paper.
\begin{lemma}\label{lemma1}
	Let ~$\Xi(\cdot)$, $\mathcal{L}(\cdot)$ and $\Pi(\cdot)$ be defined (respectively) as  in \eqref{xii}$_1$, \eqref{xii}$_2$ and \eqref{nonlinear}.
	Then the functional $\Xi:C^1(\overline I)\rightarrow C^0(\overline I)$ is locally Lipschitz, and there exists $C>0$, depending on the cable parameters, such that
	\begin{equation}\label{L}
		|\mathcal{L}(v)-\mathcal{L}(z)|\leq \|u_x-z_x\|_{L^1(I)}\qquad \forall v,z\in W^{1,1}(I)
	\end{equation}
	\begin{equation}\label{Pi}
		|\Pi(v)-\Pi(z)|\leq C\|v_x-z_x\|_{L^1(I)}\qquad \forall v,z\in W^{1,1}(I).
	\end{equation}
\end{lemma}
\begin{proof}
	Given $v,z\in C^1[0,\pi]$, we apply the Lagrange theorem, so that there exists $\varrho\in(u_x,z_x)$ such that 
	\begin{equation*}
		|\Xi(u)-\Xi(z)|=\frac{|\varrho||u_x-z_x|}{\sqrt{1+\varrho^2}}\leq |u_x-z_x|.
	\end{equation*}
	Since 
	$
	\mathcal{L}(u)-\mathcal{L}(z)=\int_0^\pi \big(\Xi(u)-\Xi(z)\big)\,dx
	$
	\eqref{L} follows.
	
	Recalling that $\xi_0$ is bounded, we infer the existence of $C>0$ such that
	\begin{equation*}
		\begin{split}
			|\Pi(v)-\Pi(z)|&=\bigg|\dfrac{b}{2}\big[\mathcal{L}\big(u\big)-\mathcal{L}_0\big]^2-\dfrac{b}{2}\big[\mathcal{L}\big(z\big)-\mathcal{L}_0\big]^2+ c\int_0^\pi \xi_0[\Xi(u)-\Xi(z)]\,dx\bigg|\\&\leq \dfrac{b}{2}\Big|\big[\mathcal{L}\big(u\big)-\mathcal{L}_0\big]^2-\big[\mathcal{L}\big(z\big)-\mathcal{L}_0\big]^2\Big|+c\max_{x\in \overline I}|\xi_0|\int_0^\pi \big|\Xi\big(u\big)-\Xi\big(z\big)\big|\,dx\\&\leq \dfrac{b}{2}\big|\mathcal{L}(u)-\mathcal{L}(z)\big|\big|\mathcal{L}(u)+\mathcal{L}(z)-2\mathcal{L}_0\big|+C\|u_x-z_x\|_{L^1(I)}\\
			&\leq C\|u_x-z_x\|_{L^1(I)}.
		\end{split}
	\end{equation*}
\end{proof}
\begin{lemma}\label{lemmah}
	Let $h(\cdot)$ and $\Pi(\cdot)$ be respectively as in \eqref{xii}$_3$ and \eqref{Pi}, then there exist $C_c,\overline C_c,c_c,\overline c_c>0$, depending on the cable parameters,  such that
	\begin{equation}\label{h-weak}
		\big(h(u),u_x\big)_0\leq -\Pi(u)+C_c\|u_x\|_{L^1(I)}+\overline C_c\qquad\forall u\in W^{1,1}(I)
	\end{equation}
	and
	\begin{equation}\label{h-L2}
		\|h(u)\|_{0}^2\leq c_c\|u_x\|_{L^1(I)}^2+\overline c_c\qquad\forall u\in W^{1,1}(I).
	\end{equation}
\end{lemma}
\begin{proof}
	Let us begin with \eqref{h-weak}. From \eqref{xii} we write 
	\begin{equation*}
		\begin{split}
			\big(h(u),u_x\big)_0=&-\int_0^\pi\bigg(b\big(\mathcal{L}(u)-\mathcal{L}_0\big)+c\,\xi_0(x)\,\bigg) \frac{\big(u+s\big)_x}{\Xi(u)}u_xdx\\
			&-\int_0^\pi\bigg(b\big(\mathcal{L}(u)-\mathcal{L}_0\big)+c\,\xi_0(x)\,\bigg) \frac{[\big(u+s\big)_x]^2-(u+s)_xs_x}{\Xi(u)}dx\\
			&-\int_0^\pi\bigg(b\big(\mathcal{L}(u)-\mathcal{L}_0\big)+c\,\xi_0(x)\,\bigg) \bigg(\Xi(u)-\frac{1+(u+s)_xs_x}{\Xi(u)}\bigg)dx.
		\end{split}
	\end{equation*}
	Adding and subtracting $b\mathcal{L}_0\big(\mathcal{L}(u)-\mathcal{L}_0\big)$ and $c\int_0^\pi\xi_0^2dx$ we obtain
	\begin{equation*}
		\begin{split}
			\big(h(u),u_x\big)_0=
			&-b\big(\mathcal{L}(u)-\mathcal{L}_0\big)\bigg(\int_0^\pi\Xi(u)-\mathcal{L}_0\bigg)-b\mathcal{L}_0\big(\mathcal{L}(u)-\mathcal{L}_0\big)-c\int_0^\pi\xi_0 \big(\Xi(u)-\xi_0\big)dx\\&-c\int_0^\pi\xi_0^2dx
			+\int_0^\pi\bigg(b\big(\mathcal{L}(u)-\mathcal{L}_0\big)+c\xi_0\bigg) \frac{1+(u+s)_xs_x}{\Xi(u)}dx\\
			=&-\Pi(u)-\dfrac{b}{2}\big(\mathcal{L}(u)-\mathcal{L}_0\big)^2-c\int_0^\pi\xi_0^2dx-b\mathcal{L}_0\big(\mathcal{L}(u)-\mathcal{L}_0\big)
			\\&+\int_0^\pi\bigg(b\big(\mathcal{L}(u)-\mathcal{L}_0\big)+c\xi_0\bigg) \frac{1+(u+s)_xs_x}{\Xi(u)}dx
			\\
			\leq&-\Pi(u)+b\mathcal{L}_0^2
			+\bigg|\int_0^\pi\bigg(b\big(\mathcal{L}(u)-\mathcal{L}_0\big)+c\xi_0\bigg) \frac{1+(u+s)_xs_x}{\Xi(u)}dx\bigg|.
		\end{split}
	\end{equation*}
	By the definition of $\Xi(u)$ in \eqref{xii} we have $1/\Xi(u)\leq 1$ and  $|(u+s)_x|/\Xi(u)\leq 1$ so that
	we may bound the term
	\begin{equation*}
		\begin{split}
			\bigg|\int_0^\pi\bigg(b\big(\mathcal{L}(u)-\mathcal{L}_0\big)+c\xi_0\bigg) \frac{1+(u+s)_xs_x}{\Xi(u)}dx\bigg|&\leq (b\big|\mathcal{L}(u)-\mathcal{L}_0\big|+c\max_{x\in \overline I}|\xi_0|)\int_0^\pi\frac{1}{\Xi(u)}+\frac{|(u+s)_x||s_x|}{\Xi(u)}dx\\&
			\leq (b\|u_x\|_{L^1(I)}+c\max_{x\in \overline I}|\xi_0|)\bigg(\pi+\int_0^\pi|s_x|dx\bigg);
		\end{split}
	\end{equation*}
	 in the last inequality we used \eqref{L}.
	The inequality \eqref{h-weak} follows taking $C_c:=b(\pi+\int_0^\pi|s_x|dx)$ and $\overline C_c:=c\max_{ \overline I}|\xi_0|(\pi+\int_0^\pi|s_x|dx)+b\mathcal{L}_0^2$.
	
	We now establish \eqref{h-L2}.
	We write
	\begin{equation*}
		\begin{split}
			\|h(u)\|_{0}^2=\int_0^\pi\bigg(b\big(\mathcal{L}(u)-\mathcal{L}_0\big)+c\,\xi_0(x)\,\bigg)^2 \frac{[(u+s)_x]^2}{\Xi(u)^2}dx\leq&
			2\bigg(b^2\pi\big(\mathcal{L}(u)-\mathcal{L}_0\big)^2+c^2\int_0^\pi\xi_0^2dx\bigg)\\\leq&
			2\bigg(b^2\pi\|u_x\|_{L^1(I)}^2+c^2\int_0^\pi\xi_0^2dx\bigg),
		\end{split}
	\end{equation*}
	where we used Young's inequality and \eqref{L}. The inequality \eqref{h-L2} follows taking $c_c:=2b^2\pi$ and $\overline c_c:=2c^2\int_0^\pi\xi_0^2dx$.	
\end{proof}

\subsection{Proof of Theorem \ref{teo1}}\label{proof-teo1}
In the calculations and estimates that follow, we must  test the equations with $w_t$ or $\theta_t$. This can be justified by density, operating on smooth solutions and passing to the limit. However, we recall a general result from \cite{temam} for second order (in time) systems, which circumvents the lack of regularity for weak solutions. 
\begin{lemma}\cite[Lemma 4.1]{temam}\label{lemmatemam}
	Let $(V,H,V')$ be a Hilbert triplet. Let $a(\cdot,\cdot)$ be a continuous and coercive bilinear form on $V$, associated to an operator $A$ continuously extending from $V$ to $V'$. This is to say  that $a(u,v)=\,_{V'}\langle  Au,v \rangle_{V}$ for all $u,v\in V$. 
	
	If $w$ is such that
	$$
	w\in L^2((0,T), V)\quad w_t\in L^2((0,T),H)\quad w_{tt}+Aw\in L^2(0,T,H)
	$$
	then, after a modification on a set of measure zero, $w\in C^0([0,T];V)$, $w_t\in C^0([0,T]; H)$ and, in the sense of distributions on $(0,T)$,
	$$
	( w_{tt}+Aw,w_t)_H=\dfrac{1}{2}\dfrac{d}{dt}\big(\|w_t\|_0^2+a(w,w)\big).
	$$
\end{lemma}
We split the proof of Theorem \ref{teo1} in different steps concerning: existence of weak solutions, uniqueness of weak solutions, existence of strong solutions and the energy identity.

\textit{$\bullet$ Existence of weak solutions.} 	
The existence of a weak solution is proved applying the Galerkin scheme, see e.g. \cite{befaga,bongazmor,falocchi1,fergazmor}.  We give  a complete proof here; for additional details, see also \cite[Theorem 1]{falocchi1} where a similar cable-hanger nonlinearity  is considered.

We denote by $\{e_{k}\}^{\infty}_{k=1}$ an orthogonal basis of $L^2(I)$, $H^1_{0}(I)$, $H^2\cap H^1_{0}(I)$, given by $	e_{k}(x)=\sqrt{2/\pi}\sin(k x)$;
%	\begin{equation*}
	% \hspace{5mm}|| e_{k}||_{2}=1,\hspace{5mm}||e_{k}||_{H^1}=k,\hspace{5mm}|| e_{k}||_{H^2}=k^2,
	%	\end{equation*}
then for any $n\geq 1$, we introduce the space	$E_n:={\rm span} \{e_1,\dots,e_n\}$.
We put for any $n\geq 1$
\begin{equation*}
	\begin{split}
		&w^0_n:=\sum_{k=1}^{n} (w_0,e_k)_0 \hspace{1mm}e_k =\sum_{k=1}^{n}\dfrac{(w_0,e_k)_{2}}{k^4} \hspace{1mm}e_k, \qquad\theta^0_n:=\sum_{k=1}^{n} (\theta_0,e_k)_0 \hspace{1mm}e_k =\sum_{k=1}^{n}\dfrac{(\theta_0,e_k)_{2}}{k^4}\hspace{1mm}e_k,\\& w^1_n:=\sum_{k=1}^{n} (w_1,e_k)_0 \hspace{1mm}e_k,  \hspace{42mm} \theta^1_n:=\sum_{k=1}^{n}(\theta_1,e_k)_0 \hspace{1mm}e_k, 
	\end{split}
\end{equation*}
so that
\begin{equation*}
	\begin{split}
		&w^0_n\rightarrow w_0 \hspace{2mm}{\rm in}\hspace{2mm} H^2(I), \hspace{5mm} \theta^0_n\rightarrow\theta_0 \hspace{2mm}{\rm in}\hspace{2mm} H^2(I),\hspace{5mm}
		w^1_n\rightarrow w_1 \hspace{2mm}{\rm in}\hspace{2mm} L^2(I), \hspace{5mm}\theta^1_n\rightarrow\theta_1 \hspace{2mm}{\rm in}\hspace{2mm} L^2(I)
	\end{split}
	\label{converg ics}
\end{equation*}
as $n\rightarrow\infty$. For any $n\geq 1$ we seek $(w_n,\theta_n)$ such that
\begin{equation*}
	w_n(x,t)=\sum_{k=1}^{n} w_n^k(t) \hspace{1mm}e_k, \hspace{5mm} \theta_n(x,t)=\sum_{k=1}^{n} \theta_n^k(t) \hspace{1mm}e_k,
\end{equation*}
which solves the problem \eqref{weak}---thus we need to solve for the Fourier coefficients $\{w_n^k(t),\theta_n^k(t)\}_{k=1}^n$. 
 Restricting the test functions $v, \varphi\in E_n$, \eqref{weak} becomes for each $j=1,2,...,n$.
\begin{equation}\small
	\begin{cases}
		\big((w_n)_{tt},e_j\big)_0+\mu\big((w_n)_t,e_j\big)_0+\big(w_n,e_j\big)_{2}+\big[S\|w_n\|^2_1-P\big](w_n,e_j)_1-\big(f(w_n,\theta_n),e'_j\big)_0=\big(g-\beta \Upsilon(\theta_n)_t-\eta\theta_n,e_j\big)_0\vspace{2mm}\\
		\frac{\ell^2}{3}\big((\theta_n)_{tt},e_j\big)_0+\zeta \big((\theta_n)_t,e_j\big)_0+\epsilon\big(\theta_n,e_j\big)_{2}+\kappa\big(\theta_n,e_j\big)_{1}-\big(\overline f(w_n,\theta_n),e_j'\big)_0=0.
	\end{cases}
	\label{eq weak sist gal}
\end{equation}
Using orthogonality of the basis $\{e_k\}_{k=1}^{\infty}$, we obtain the system
\begin{equation}\small
	\begin{cases}
		\ddot{w}^k_n(t)+\mu\,\dot{w}^k_n(t)+k^4 w^k_n(t)+\big[S\big(\sum_{r=1}^{n}r^2w^r_n(t)^2\big)-P\big]k^2w^k_n(t)+\beta \Upsilon\dot{\theta}^k_n(t)+\eta \theta_n^k(t) =\big(f(w^k_n,\theta^k_n),e'_k\big)_0+\big(g,e_k\big)_0\vspace{2mm}\\
	\frac{\ell^2}{3}\ddot{\theta}^k_n(t)+\zeta \,\dot{\theta_n^k}(t)+(\epsilon k^2+\kappa )k^2\,\theta^k_n(t)=\big(\overline f(w^k_n,\theta^k_n),e'_k\big)_0\vspace{2mm}\\
		w^k_n(0)=(w_0,e_k)_0 \hspace{10mm}\dot{w}^k_n(0)=(w_1,e_k)_0\\\theta^k_n(0)=(\theta_0,e_k)_0\hspace{13mm}\dot{\theta}^k_n(0)=(\theta_1,e_k)_0\hfill\forall k=1,\dots,n.
	\end{cases}
	\label{eq weak sist lin gal  2n}
\end{equation}
Since $f(\cdot,\cdot)$ and  $\overline f(\cdot,\cdot)$ are scalar functions continuous in their arguments, see \eqref{melan42}, and $g$ is constant in time, from the standard theory of ODEs this finite-dimensional system admits a local solution, defined on some  $[0,t_n)$ with $t_n\in(0,T]$.

For simplicity, in the sequel we often denote the time partial derivative of a function by $\dot{v}$, instead of $v_t$. 
The obtained solution $(w_n,\theta_n)$  is $C^2([0,t_n),E_n)$; therefore, testing the first equation in \eqref{eq weak sist gal} by $\dot{w}_n\in E_n$, the second by $\dot{\theta}_n\in E_n$, summing the equations and integrating over $t\in (0,t_n)$ we obtain
\begin{equation}
	\begin{split}
		&\dfrac{\|\dot{w}_n(t)\|^2_0}{2}+\dfrac{\|w_n(t)\|^2_{2}}{2}+\frac{\ell^2}{6}\|\dot{\theta}_n(t)\|^2_0+\epsilon\dfrac{\|\theta_n(t)\|^2_{2}}{2}+\kappa\dfrac{\|\theta_n(t)\|^2_{1}}{2}-P\dfrac{\|w_n(t)\|^2_{1}}{2}+S\dfrac{\|w_n(t)\|^4_{1}}{4}\\&\Pi(w_n(t)+\ell\theta_n(t)\big)+\Pi(w_n(t)-\ell\theta_n(t)\big)+\mu \int_0^t\|\dot{w}_n(s)\|_0^2\,ds+\zeta \int_0^t\|\dot{\theta}_n(s)\|_0^2\,ds\\=&~\dfrac{\|w^1_n\|^2_0}{2}+\dfrac{\|w^0_n\|^2_{2}}{2}+\frac{\ell^2}{6}\|\theta^1_n\|^2_0+\epsilon\dfrac{\|\theta^0_n\|^2_{2}}{2}+\kappa\dfrac{\|\theta^0_n\|^2_{1}}{2}-P\dfrac{\|w_n^0\|^2_{1}}{2}+S\dfrac{\|w_n^0\|^4_{1}}{4}+\\&\Pi\big(w^0_n+\ell\theta_n^0\big)+\Pi\big(w^0_n-\ell\theta_n^0\big)+\int_0^t\big(g,\dot{w}_n(s)\big)_0\,ds-\int_0^t\big(\beta \Upsilon\dot{\theta}_n(s)+\eta \theta_n(s),\dot{w}_n(s)\big)_0\,ds.
	\end{split}
	\label{step 2.1}
\end{equation} 
We apply the H\"older and Young inequalities to the terms 
\begin{align}\label{bd1}
		\bigg|\int_0^t\big(g,\dot{w}_n(s)\big)_0\,ds\bigg|\leq&~ \dfrac{1}{2}\bigg(t\|g\|_{0}^2+\int_0^t\|\dot{w}_n(s)\|^2_0\,ds\bigg)\\
		\bigg|\int_0^t\big(\beta \Upsilon\dot{\theta}_n(s)+\eta \theta_n(s),\dot{w}_n(s)\big)_0\,ds\bigg|\leq&~ \dfrac{\beta\ell+|\eta|}{2}\int_0^t(\|\dot{\theta}_n(s)\|^2_0+\|\theta_n(s)\|^2_0+2\|\dot{w}_n(s)\|^2_0)\,ds.\label{bd11}
\end{align}
Combining these inequalities, considering the regularity of the initial conditions, and noting the fact that $\tau\mapsto\frac{P}{2}\tau^2-\frac{S}{4}\tau^4$ has the maximum $\frac{P^2}{4S}$, we obtain
\begin{equation}
	\begin{split}
		&\dfrac{\|\dot{w}_n(t)\|^2_0}{2}+\dfrac{\|w_n(t)\|^2_{2}}{2}+\frac{\ell^2}{6}\|\dot{\theta}_n(t)\|^2_0+\epsilon\dfrac{\|\theta_n(t)\|^2_{2}}{2}+\kappa\dfrac{\|\theta_n(t)\|^2_{1}}{2}+\\&\Pi(w_n(t)+\ell\theta_n(t)\big)+\Pi(w_n(t)-\ell\theta_n(t)\big)+\mu \int_0^t\|\dot{w}_n(s)\|_0^2\,ds+\zeta \int_0^t\|\dot{\theta}_n(s)\|_0^2\,ds\\\leq&~\dfrac{\|w_1\|^2_2}{2}+\dfrac{\|w_0\|^2_{2}}{2}+\frac{\ell^2}{6}\|\theta_1\|^2_0+\epsilon\dfrac{\|\theta_0\|^2_{2}}{2}+\kappa\dfrac{\|\theta_0\|^2_{1}}{2}+S\dfrac{\|w_0\|_1^4}{4}+\Pi\big(w_0+\ell\theta_0\big)+\Pi\big(w_0-\ell\theta_0\big)+\\&\dfrac{\beta\ell+|\eta|}{2}\int_0^t(\|\dot{\theta}_n(s)\|^2_0+\|\theta_n(s)\|^2_0+2\|\dot{w}_n(s)\|^2_0)\,ds+\dfrac{1}{2}\int_0^t\|\dot{w}_n(s)\|^2_0\,ds+\dfrac{T}{2}\|g\|_{0}^2+\dfrac{P^2}{4S}.
	\end{split}
	\label{step 2.2}
\end{equation}

Observing that $\Pi(w_n+\ell\theta_n)+\Pi(w_n-\ell\theta_n)+2c\int_0^\pi\xi_0^2\,dx>0$,  applying \eqref{spectral} and \eqref{Pi} with $v=w_0\pm\ell\theta_0$ and $z=0$, we infer the existence of $C,c>0$ such that	
\begin{equation*}
	\begin{split}
		\|\dot{w}_n(t)\|^2_0+\|w_n(t)\|^2_{2}+\|\dot{\theta}_n(t)\|^2_0+\|\theta_n(t)\|^2_{2}\leq&~ C\big(\|w_1\|^2_0+\|w_0\|^2_{2}+\|\theta_1\|^2_0+\|\theta_0\|^2_{2}+\|w_0\|^4_{2}+1+T\big)\\&+c\int_0^t(\|\dot{w}_n(s)\|^2_0+\|w_n(s)\|^2_{2}+\|\dot{\theta}_n(s)\|^2_0+\|\theta(s)\|^2_{2})\,ds.
	\end{split}
	\label{step 2.3}
\end{equation*} 
Hence, by the Gronwall inequality, we obtain
\begin{equation}\label{stima}
	\|\big(w_n(t),\dot{w}_n(t);\theta_n(t),\dot{\theta}_n(t)\big)\|_Y^2\leq C\big(\|(w_0,w_1;\theta_0,\theta^1)\|_Y^2+\|w_0\|^4_{2}+1+T\big)e^{cT},
\end{equation}
for any	$t\in[0,t_n)$, where the constants $C$ and $c$ are now {\em independent of $n$ and $t$}. 
Hence,	$w_n$ and $\theta_n$ are globally defined in $\R_+$ for every $n\geq1$ and the sequences $\{w_n\}_n$ and $\{\theta_n\}_n$ are uniformly bounded in the space $C^0([0,T]; (H^2\cap H_0^1)(I))\cap C^1([0,T], L^2(I))$ for each finite $T>0$. We now show that $\{w_n\}_n$ and $\{\theta_n\}_n$ admit strongly convergent subsequences in the same spaces.

The estimate \eqref{stima} shows that $\{w_n\}_n$ and $\{\theta_n\}_n$ are bounded and equicontinuous in $C^0([0,T]; L^2(I))$; then, by the Ascoli--Arzel\`a theorem, we conclude that, up to a subsequence, $u_n \rightarrow  u$ strongly in $C^0([0,T]; L^2(I))$.

For every $n>m\geq 1$, we consider the two equations in \eqref{eq weak sist gal} satisfied  by $(w_n,\theta_n)$ and tested respectively by $\dot{w}_n$ and $\dot{\theta}_n$;  afterwards, we consider the two equations in \eqref{eq weak sist gal} satisfied  by $(w_m,\theta_m)$ and tested respectively by $\dot{w}_m$ and $\dot{\theta}_m$. We then subtract \eqref{eq weak sist gal}$_m$ from \eqref{eq weak sist gal}$_n$ and we put $w_{n,m}:=w_n-w_m$ and $\theta_{n,m}:=\theta_n-\theta_m$. Summing and integrating in time, we find the estimate for all $t\in[0,T]$
\begin{equation}
\begin{split}
		&\dfrac{\|\dot{w}_{n,m}(t)\|^2_0}{2}+\dfrac{\|w_{n,m}(t)\|^2_{2}}{2}+\frac{\ell^2}{6}\|\dot{\theta}_{n,m}(t)\|^2_0+\epsilon\dfrac{\|\theta_{n,m}(t)\|^2_{2}}{2}+\kappa\dfrac{\|\theta_{n,m}(t)\|^2_{1}}{2}\\&+\mu \int_0^t\|\dot{w}_{n,m}(s)\|_0^2\,ds+\zeta \int_0^t\|\dot{\theta}_{n,m}(s)\|_0^2\,ds\\ =&~P\dfrac{\|w_{n,m}(t)\|^2_{1}}{2}-\frac{S}{4}\bigg(\|w_{n}(t)\|^4_1-\|w_m(t)\|^4_1\bigg)+\dfrac{\|w^1_{n,m}\|^2_0}{2}+\dfrac{\|w^0_{n,m}\|^2_{2}}{2}\\&+\frac{\ell^2}{6}\|\theta^1_{n,m}\|^2_0+\epsilon\dfrac{\|\theta^0_{n,m}\|^2_{2}}{2}+\kappa\dfrac{\|\theta^0_{n,m}\|^2_{1}}{2}-P\dfrac{\|w^0_{n,m}\|^2_{1}}{2}+\frac{S}{4}\bigg(\|w_n^0\|^4_1-\|w_m^0\|^4_1\bigg)\\&+\int_0^t\big(g,\dot{w}_{n,m}(s)\big)_0\,ds-\int_0^t\big(\beta \Upsilon\dot{\theta}_{n,m}(s)+\eta \theta_{n,m}(s),\dot{w}_{n,m}(s)\big)_0\,ds+\\&-\Pi(w_n(t)+\ell\theta_n(t))+\Pi(w_m(t)+\ell\theta_m(t))+\Pi(w_n^0+\ell\theta_n^0)-\Pi(w_m^0+\ell\theta_m^0)\\&-\Pi(w_n(t)-\ell\theta_n(t))+\Pi(w_m(t)-\ell\theta_m(t))+\Pi(w_n^0-\ell\theta_n^0)-\Pi(w_m^0-\ell\theta_m^0)
		.
		\end{split}
	\label{step 2.6}
\end{equation} 
We apply Lemma \ref{lemma1} with $v=w_n+\ell\theta_n$ and $z=w_m+\ell\theta_m$, so that through \eqref{interpolation} we infer the existence of $C>0$ such that
\begin{equation*}
	\begin{split}
		|\Pi(w_n(t)+\ell\theta_n(t))-\Pi(w_m(t)+\ell\theta_m(t))|&\leq C\big(\|w_n(t)-w_m(t)\|_{1}+\|\theta_n(t)-\theta_m(t)\|_{1}\big)\\&\leq C\big(\|w_{n,m}(t)\|^{1/2}_{0}+\|\theta_{n,m}(t)\|^{1/2}_{0}\big)
	\end{split}
\end{equation*}
and similarly for each of the terms involving $\Pi(\cdot)$.
Again through \eqref{interpolation} we infer
\small{
	\begin{equation*}
		\begin{split}
			\frac{P}{2}\|w_{n,m}(t)\|^2_{1}-\frac{S}{4}\big(\|w_{n}(t)\|^4_1-\|w_m(t)\|^4_1\big)=&~
			\frac{P}{2}\|w_{n,m}(t)\|^2_{1}-\frac{S}{4}(\|w_{n}(t)\|^2_1-\|w_m(t)\|^2_1)(\|w_{n}(t)\|^2_1+\|w_m(t)\|^2_1)\\\leq &~ C\big(\|w_{n,m}(t)\|_0+\|w_{n,m}(t)\|_1\big) \\ \leq& ~ C\big(\|w_{n,m}(t)\|_{0}+\|w_{n,m}(t)\|^{1/2}_{0}\big),
		\end{split}
\end{equation*}}
\normalsize	and similarly for ~$-\frac{P}{2}\|w_{n,m}^0\|^2_{1}+\frac{S}{4}(\|w_n^0\|^4_1-\|w_m^0\|^4_1)$.
As in \eqref{bd1}, we bound
\begin{equation*}
	\begin{split}
&\bigg|\int_0^t\big(\beta \Upsilon\dot{\theta}_{n,m}(s)+\eta \theta_{n,m}(s),\dot{w}_{n,m}(s)\big)_2\,ds\bigg|\leq \dfrac{\beta\ell+|\eta|}{2}\int_0^t(\|\dot{\theta}_{n,m}(s)\|^2_0+\|\theta_{n,m}(s)\|^2_0+2\|\dot{w}_{n,m}(s)\|^2_0)\,ds;
	\end{split}
\end{equation*}
from \eqref{stima} we know that $\{\dot{w}_{n}(s)\}_n$ is bounded in $C^0([0,T];L^2(I))$, and since  $g\in L^2(I)$, denoting by $P_n$ (resp. $P_m$) its projection on the first $n$ (resp. $m$) modes we obtain 
$$
\bigg|\int_0^t\big(g,\dot{w}_{n,m}(s)\big)_0\,ds\bigg|\leq \|\dot{w}_{n,m}\|_{C^0([0,T];L^2(I))}\int_0^t\|P_ng-P_mg\|_0\,ds\leq CT\|P_ng-P_mg\|_0.
$$
Combing all these inequalities, from \eqref{step 2.6} we infer the existence of $c, C>0$ such that
\begin{equation*}
	\begin{split}
		&\hspace{-10mm}\|\dot{w}_{n,m}(t)\|^2_0+\|w_{n,m}(t)\|^2_{2}+\|\dot{\theta}_{n,m}(t)\|^2_0+\|\theta_{n,m}(t)\|^2_{2}\\\leq C \Big\{&\|w^1_{n,m}\|^2_0+\|w^0_{n,m}\|^2_{2}+\|\theta^1_{n,m}\|^2_0+\|\theta^0_{n,m}\|^2_{2}+\|w^0_{n,m}\|_0+\|w^0_{n,m}\|_0^{1/2}+\|\theta^0_{n,m}\|_0^{1/2}\\&+\|w_{m,n}\|_{C^0([0,T];L^2(I))}+\|w_{n,m}\|^{1/2}_{C^0([0,T];L^2(I))}+\|\theta_{n,m}\|^{1/2}_{C^0([0,T];L^2(I))}+T\|P_ng-P_mg\|_0\Big\}\\&+c\int_0^t(\|\dot{w}_{n,m}(s)\|^2_0+\|w_{n,m}(s)\|^2_{2}+\|\dot{\theta}_{n,m}(s)\|^2_0+\|\theta_{n,m}(s)\|^2_{2})\,ds.
	\end{split}
\end{equation*}
Applying Grownwall Lemma we infer
\begin{align*}
		\|\big(w_{n,m}(t),\dot{w}_{n,m}(t);\theta_{n,m}(t),\dot{\theta}_{n,m}(t)\big)\|_Y^2\leq&~ C\Big\{\|(w^0_{n,m},w^1_{n,m};\theta^0_{n,m},\theta^1_{n,m})\|_Y^2+\|w^0_{n,m}\|_0\\&+\|w^0_{n,m}\|_0^{1/2}+\|\theta^0_{n,m}\|_0^{1/2}+\|w_{m,n}\|_{C^0([0,T];L^2(I))}\\&+\|w_{n,m}\|^{1/2}_{C^0([0,T];L^2(I))}+\|\theta_{n,m}\|^{1/2}_{C^0([0,T];L^2(I))}\\&+
		T\|P_ng-P_mg\|_0\Big\}e^{cT}\rightarrow0\quad \text{as }n,m\rightarrow\infty,
\end{align*}
thanks to the strong convergence in $C^0([0,T];L^2(I))$. Therefore $\{w_n\}_n$ and $\{\theta_n\}_{n}$ are Cauchy sequences in the space $C^0([0,T]; (H^2\cap H_0^1)(I))\cap C^1([0,T],L^2(I))$; in turn, this yields, up to a subsequence, 
\begin{equation}\label{conv}
	w_n\rightarrow w,\quad \theta_n\rightarrow\theta \quad \text{in }C^0([0,T]; (H^2\cap H_0^1)(I))\cap C^1([0,T],L^2(I)).
\end{equation}

%	In classical way from the estimate \eqref{stima} we infer the boundedness of $\{w_n\}, \{\theta_n\}$ and $\{\dot{w}_n\}, \{\dot{\theta}_n\}$ in $L^\infty((0,T),H^2(I))$ and resp. in $L^\infty((0,T),L^2(I))$. Then, it is possible to extract  subsequences, so that there is  weak* convergence in the previous spaces, see \cite{falocchi1} for details. Then, due to the compact embedding $H^1(I_T)\subset L^2(I_T)$, we obtain the strong convergence
%	\begin{equation*}
	%	w_{n}  \rightarrow w, \hspace{3mm} \theta_{n}\rightarrow\theta \hspace{12mm} {\rm in} \hspace{10mm} L^{2}(I_T).
	%	\end{equation*}

Now take $v,\varphi\in (H^2\cap H_0^1)(I)$ and consider the sequences of projections $P_nv$ and $P_n\varphi$. Taking $P_nv$ and $P_n\varphi$ as test functions in \eqref{eq weak sist gal}, multiplying by $\psi\in C^\infty_c(0,T)$ and integrating over $[0,T]$ we get
\begin{equation}
	\begin{cases}
		\int_0^T\big((w_n)_{t},P_nv\big)_0\psi'=\int_0^T\big\{\big(\mu(w_n)_t+\beta \Upsilon(\theta_n)_t+\eta\theta_n-g,P_nv\big)_0+\big(w_n,P_nv\big)_{2}-\big(f(w_n,\theta_n),P_nv'\big)_0\\\hspace{44mm}+\big[S\|w_n\|_1^2-P\big]\big(w_n,P_nv\big)_{1}\big\}\psi\\
		\int_0^T\frac{\ell^2}{3}\big((\theta_n)_{t},P_n\varphi\big)_0\psi'=\int_0^T\big[\zeta \big((\theta_n)_t,P_n\varphi\big)_0+\epsilon\big(\theta_n,P_n\varphi\big)_{2}+\kappa\big(\theta_n,P_n\varphi\big)_{1}-\big(\overline f(w_n,\theta_n),P_n\varphi'\big)_0\big]\psi.
	\end{cases}
	\label{eqweak}
\end{equation}
To pass to the limit in \eqref{eqweak} we must consider the nonlinear terms. In particular, we have
\begin{equation}\label{S}
	S(\|w_n(t)\|_1^2-\|w(t)\|^2_1)\leq C\|w_n(t)-w(t)\|_1\leq C\|w_n-w\|_{C^0([0,T];H^1(I))}\rightarrow 0\quad \text{ as }n\rightarrow\infty;
\end{equation}
concerning cable nonlinearities, $f(w,\theta)$ and $\overline f(w_n,\theta_n)$,
from \eqref{L} and \eqref{conv} we have 
\begin{equation*}
	\begin{split}
		|\mathcal{L}(w_n\pm\ell\theta_n)-\mathcal{L}(w\pm\ell\theta)|&\leq C\big(\|w_n(t)-w(t)\|_{1}+\|\theta_n(t)-\theta(t)\|_{1}\big)\\&\leq C\big(\|w_n-w\|_{C^0([0,T];H^1(I))}+\|\theta_n-\theta\|_{C^0([0,T];H^1(I))}\big)\rightarrow 0\quad \text{ as }n\rightarrow\infty,
	\end{split}
\end{equation*}
implying $\mathcal{L}(w_n\pm\ell\theta_n)\rightarrow\mathcal{L}(w\pm\ell\theta)$ as $n\rightarrow\infty$.
Moreover, we obtain 
\begin{equation*}
	\int_{0}^{T}\int_{0}^{\pi}\dfrac{[(w_n\pm\ell\theta_n+s)_{x}]^2}{\Xi(w_n\pm\ell\theta_n)^2}dx dt< T\pi.
\end{equation*}
Hence, $f(w_n,\theta_n)$ and $\overline f(w_n,\theta_n)$, being continuous in their arguments, converge weakly, up to a subsequence, to $f(w,\theta)$ and $\overline f(w,\theta)$  in $L^2(I_T)$. Therefore, it is possible to pass to the limit the equations \eqref{eqweak}. We then obtain, by rewriting, that
\begin{equation*}
	\begin{cases}
		\int_0^T\big(w_{t},v\big)_0\psi'=\int_0^T\big\{\big(\mu w_t+\beta \Upsilon\theta_t+\eta\theta-g,v\big)_0+\big(w,v\big)_{2}+\big[S\|w\|_1^2-P\big]\big(w,v\big)_{1}-\big(f(w,\theta),v'\big)_0\big\}\psi\\
		\int_0^T\frac{\ell^2}{3}\big(\theta_{t},\varphi\big)_0\psi'=\int_0^T\big[\zeta \big(\theta_t,\varphi\big)_0+\epsilon\big(\theta,\varphi\big)_{2}+\kappa\big(\theta,\varphi\big)_{1}-\big(\overline f(w,\theta),\varphi'\big)_0\big]\psi.
	\end{cases}
	\label{eqweak2}
\end{equation*}
For our choice of test functions, we deduce that  $w_{tt},\theta_{tt}\in C^0([0,T];\mathcal{H})$ and they solve $a.e.~t \in (0,T)$
\begin{equation}\label{wtt}
	\begin{cases}
		w_{tt}=-Lw-\mu w_t+\big(S\|w\|_1^2-P\big)w_{xx}-[f(w,\theta)]_x+g-\beta \Upsilon\theta_t-\eta\theta\\
		\frac{\ell^2}{3}\theta_{tt}=-\overline L\theta-\zeta \theta_t-[\overline f(w,\theta)]_x
	\end{cases}
\end{equation}
where $L, \overline L:(H^2\cap H_0^1)(I)\rightarrow \mathcal{H}$ stand for the canonical Riesz isometric isomorphisms respectively given by $\,_{-2}\langle Lw,v\rangle_{2}:=(w,v)_{2}$ for all $w,v\in (H^2\cap H^1_0)(I)$ and $	\,_{-2}\langle \overline L\theta,\varphi\rangle_{2}:=\kappa(\theta,\varphi)_{1}+\epsilon(\theta,\varphi)_{2}$ for all $\theta,\varphi\in (H^2\cap H^1_0)(I)$. We thus conclude that $(w,\theta)$ is a weak solution of \eqref{weak}.

\textit{$\bullet$ Uniqueness of weak solution.} 
For contradiction, consider two solutions $(w^\mathcal{I},\theta^\mathcal{I})$, $(w^\mathcal{II},\theta^\mathcal{II})$ satisfying the same initial conditions and the regularity in Definition \ref{def_weak}. By subtracting the two systems and denoting by
$W=w^\mathcal{I}-w^\mathcal{II}$ and $\Theta=\theta^\mathcal{I}-\theta^\mathcal{II}$, we see that $(W,\Theta)$ is a solution of
\begin{equation*}
	\begin{cases}
			\,_{-2}\langle W_{tt},v \rangle_{2}+\mu (W_t,v)_0+(W,v)_{2}-P(W,v)_1+\beta \Upsilon (\Theta_t,v)_0+\eta(\Theta,v)_0\\=S\big(\|w^\mathcal{I}\|^2_1w^\mathcal{I}_{xx}-\|w^\mathcal{II}\|^2_1w^\mathcal{II}_{xx},v\big)_0-\big([f(w^\mathcal{I},\theta^\mathcal{I})-f(w^\mathcal{II},\theta^{\mathcal{II}})]_x,v\big)_0\vspace{2mm}\\
		\frac{\ell^2}{3}	\,_{-2}\langle\Theta_{tt},\varphi\rangle_{2}+\zeta (\Theta_t,\varphi)_0+\epsilon(\Theta,\varphi)_{2}+\kappa(\Theta,\varphi)_{1}=-\big([\overline f(w^\mathcal{I},\theta^\mathcal{I})-\overline f(w^\mathcal{II},\theta^\mathcal{II})]_x,\varphi\big)_0
	\end{cases}
	\label{eq weak sist2sol}
\end{equation*}
for all $v, \varphi\in (H^2\cap H_0^1)(I)$ with homogeneous initial conditions and $t>0$.

By \eqref{wtt} and Lemma \ref{lemmatemam}, we test the equations
by $v=\dot{W}$ and $\varphi=\dot{\Theta}$. Summing the equations and integrating over $(0,t)$ (omitting the dependence on $t$), we end up with
\begin{equation*}
	\begin{split}
		&\dfrac{\|\dot{W}\|^2_0}{2}+\dfrac{\|W\|^2_{2}}{2}+\dfrac{\|\dot{\Theta}\|^2_0}{2}+\epsilon\dfrac{\|\Theta\|^2_{2}}{2}+\kappa\dfrac{\|\Theta\|^2_{1}}{2}-P\dfrac{\|W\|^2_{1}}{2}+\mu \int_0^t\|\dot{W}\|_0^2+\zeta \int_0^t\|\dot{\Theta}\|_0^2\\=&-\int_0^t\big([f(w^\mathcal{I},\theta^\mathcal{I})-f(w^\mathcal{II},\theta^\mathcal{II})]_x,\dot{W}\big)_0-\int_0^t\big([\overline f(w^\mathcal{I},\theta^\mathcal{I})-\overline f(w^\mathcal{II},\theta^\mathcal{II})]_x,\dot{\Theta}\big)_0\\&+S\int_0^t\big(\|w^\mathcal{I}\|^2_1W_{xx}+[\|w^\mathcal{I}\|^2_1-\|w^\mathcal{II}\|^2_1]w^\mathcal{II}_{xx}\,,\dot{W}\big)_0-\int_0^t\big(\beta \Upsilon\dot{\Theta}+\eta \Theta,\dot{W}\big)_0.
	\end{split}
	\label{unique}
\end{equation*} 
To estimate the cable nonlinearity terms on the right hand side (see $f$ and $\overline f$ in \eqref{melan42}), we need the following inequalities: let $u^\mathcal{I},u^\mathcal{II}\in (H^2\cap H_0^1)(I)$ then, by Lemma \eqref{lemma1}, it holds
\begin{equation*}
	\begin{split}
		\bigg|\dfrac{(u^\mathcal{I}+s)_{xx}}{\Xi(u^\mathcal{I})^3}-\dfrac{(u^\mathcal{II}+s)_{xx}}{\Xi(u^\mathcal{II})^3}\bigg|=&~\bigg|\dfrac{\Xi(u^\mathcal{II})^3(u^\mathcal{I}+s)_{xx}-\Xi(u^\mathcal{I})^3(u^\mathcal{II}+s)_{xx}}{\Xi(u^\mathcal{I})^3\Xi(u^\mathcal{II})^3}\bigg|\\=&~\bigg|\dfrac{\Xi(u^\mathcal{II})^3(u^\mathcal{I}-u^\mathcal{II})_{xx}-[\Xi(u^\mathcal{I})^3-\Xi(u^\mathcal{II})^3](u^\mathcal{II}+s)_{xx}}{\Xi(u^\mathcal{I})^3\Xi(u^\mathcal{II})^3}\bigg|\\
		&\hspace{-20mm}\leq~|(u^\mathcal{I}-u^\mathcal{II})_{xx}|+|\Xi(u^\mathcal{I})-\Xi(u^\mathcal{II})|\big(\Xi(u^\mathcal{I})^2+\Xi(u^\mathcal{I})\Xi(u^\mathcal{II})+\Xi(u^\mathcal{II})^2\big)|(u^\mathcal{II}+s)_{xx}|\\
		\leq&~|(u^\mathcal{I}-u^\mathcal{II})_{xx}|+\frac{3}{2}|(u^\mathcal{I}-u^\mathcal{II})_x|\big(\Xi(u^\mathcal{I})^2+\Xi(u^\mathcal{II})^2\big)|(u^\mathcal{II}+s)_{xx}|\\
		\leq&~|(u^\mathcal{I}-u^\mathcal{II})_{xx}|+\frac{3}{2}\big(2+(u^\mathcal{I})_x^2+(u^\mathcal{II})_x^2\big)|(u^\mathcal{II}+s)_{xx}||(u^\mathcal{I}-u^\mathcal{II})_x|;
	\end{split}
\end{equation*}
due to the compact embedding $H^2(I)\subset\subset C^1(\overline I)$, applying the H\"older and Young inequalities, and choosing $u^\mathcal{I}=w^\mathcal{I}\pm\ell\theta^\mathcal{I}$ and $u^\mathcal{II}=w^\mathcal{II}\pm \ell\theta^\mathcal{II}$;
we then infer the existence of $C>0$ such that
\begin{equation*}
	\begin{split}
		\bigg|\bigg(\dfrac{(w^\mathcal{I}\pm\ell\theta^\mathcal{I}+s)_{xx}}{\Xi(w^\mathcal{I}\pm\ell\theta^\mathcal{I})^3}-&\dfrac{(w^\mathcal{II}\pm\ell\theta^\mathcal{II}+s)_{xx}}{\Xi(w^\mathcal{II}\pm\ell\theta^\mathcal{II})^3},\dot{W}\bigg)_0\bigg|\\\leq&~ \big(\|W\|_{2}+\ell\|\Theta\|_{2}+C\|(u^\mathcal{II}+s)_{xx}\|_0\|(W\pm\ell\Theta)_x\|_{C^0(\overline I)}\big)\|\dot{W}\|_{0}\\
		\leq&~ C\big(\|W\|_{2}+\|\Theta\|_{2}\big)\|\dot{W}\|_{0} \\ \leq&~ C\big(\|\dot{W}\|_{0} ^2+\|W\|^2_{2}+\|\Theta\|^2_{2}\big).
	\end{split}
\end{equation*}
From \eqref{L}, we obtain the existence of $C>0$ such that
$$
|\big(\mathcal{L}(w^\mathcal{I}\pm\ell\theta^\mathcal{I})-\mathcal{L}(w^\mathcal{II}\pm\ell\theta^\mathcal{II}),\dot{W}\big)_2|\leq C\big(\|W\|_{1}+\|\Theta\|_{1}\big)\|\dot{W}\|_0\leq C\big(\|\dot{W}\|_0^2+\|W\|_{2}^2+\|\Theta\|_{2}^2\big) ;
$$
Thanks to the boundedness of the function $\xi_0(x)$, see \eqref{xii}, and the previous estimates, we obtain
$$
\bigg|\int_0^t\big([f(w^\mathcal{I},\theta^\mathcal{I})-f(w^\mathcal{II},\theta^\mathcal{II})]_x,\dot{W}\big)_0\bigg|\leq C\int_0^t\big(\|\dot{W}\|_{0} ^2+\|W\|^2_{2}+\|\Theta\|^2_{2}\big).
$$
Similarly, we find
$$
\bigg|\int_0^t\big([\overline f(w^\mathcal{I},\theta^\mathcal{I})-\overline f(w^\mathcal{II},\theta^\mathcal{II})]_x,\dot{\Theta}\big)_0\bigg|\leq C\int_0^t\big(\|\dot{\Theta}\|_{0} ^2+\|W\|^2_{2}+\|\Theta\|^2_{2}\big).
$$
We bound the last nonlinear term as follows
$$
S\bigg|\int_0^t\big(\|w^\mathcal{I}\|^2_1W_{xx}+[\|w^\mathcal{I}\|^2_1-\|w^\mathcal{II}\|^2_1]w^\mathcal{II}_{xx}\,,\dot{W}\big)_0\bigg|\leq 
C\int_0^t(\|W\|_2+\|W\|_1)\|\dot{W}\|_0\leq 
C\int_0^t\big(\|\dot{W}\|^2_0+\|W\|_2^2\big).$$
As in \eqref{bd1}, we apply the H\"older and Young inequalities to
\begin{equation*}\label{usual}
	\bigg|\int_0^t\big(\beta \Upsilon\dot{\Theta}+\eta \Theta,\dot{W}\big)_0\bigg| \leq \dfrac{\beta\ell+|\eta|}{2}\int_0^t(\|\dot{\Theta}\|^2_0+\|\Theta\|^2_0+2\|\dot{W}\|^2_0);
\end{equation*}
collecting  these inequalities, we obtain a constant $C>0$ such that 
\begin{equation*}
	\begin{split}
		&\|\dot{W}(t)\|^2_0+\|W(t)\|^2_{2}+\|\dot{\Theta}(t)\|^2_0+\|\Theta(t)\|^2_{2}\!\leq\! C\!\int_0^t\big(\|\dot{W}(s)\|^2_0+\|W(s)\|^2_{2}+\|\dot{\Theta}(s)\|^2_0+\|\Theta(s)\|^2_{2}\big)ds.
	\end{split}
	\label{uniques}
\end{equation*}
Hence, the Gronwall inequality guarantees $(W,\Theta)\equiv (0,0)$.

\textit{$\bullet$ Strong solution.} We assume the improved regularity of the data: $w_0,\theta_0\in\mathcal{D}$ and $w_1,\theta_1\in (H^2\cap H_0^1)(I)$. Then we formally differentiate \eqref{eq weak sist gal} with respect to $t$, we take as test functions $v=\ddot{w}_n$ and $\varphi=\ddot{\theta}_n$ and summing the equations, we obtain
\begin{equation}\label{eq}
	\begin{split}
		\dfrac{1}{2}\dfrac{d}{dt}&\bigg(\|\ddot{w}_n\|_0^2+\|\dot{w}_n\|^2_{2}+\frac{\ell^2}{3}\|\ddot{\theta}_n\|_0^2+\epsilon\|\dot{\theta}_n\|^2_{2}+\kappa\|\dot{\theta}_n\|^2_{1}-P\|\dot{w}_n\|^2_{1}+\dfrac{S}{2}\|\dot{w}_n\|^4_{1}\bigg)+\mu\|\ddot{w}_n\|_0^2+\zeta\|\ddot{\theta}_n\|_0^2\\=&~-\big([f(w_n,\theta_n)]_{xt},\ddot{w}_n\big)_0-\big([\overline f(w_n,\theta_n)]_{xt},\ddot{\theta}_n\big)_0-\big(\beta \Upsilon \ddot{\theta}_n+\eta\dot{\theta}_n,\ddot{w}_n\big)_0.
	\end{split}	
\end{equation}
We omit the dependence on $t$ for brevity.
To handle the right hand side terms, we compute the derivative
$$
\frac{d}{dt}\bigg[\dfrac{(w_n\pm\ell\theta_n+s)_{xx}}{\Xi(w_n\pm\ell\theta_n)^3}\bigg]=\dfrac{(w_n\pm\ell\theta_n)_{xxt}\Xi(w_n\pm\ell\theta_n)^2-3(w_n\pm\ell\theta_n+s)_{xx}(w_n\pm\ell\theta_n+s)_{x}(w_n\pm\ell\theta_n)_{xt}}{\Xi(w_n\pm\ell\theta_n)^5}
$$
so that
\begin{equation*}
	\begin{split}
		\bigg|\bigg(\bigg[\dfrac{(w_n\pm\ell\theta_n+s)_{xx}}{\Xi(w_n\pm\ell\theta_n)^3}\bigg]_t,\ddot{w}_n\bigg)_0\bigg|\leq& C\big(\|(w_n\pm\ell\theta_n)_{xxt}\|_0+\|(w_n\pm\ell\theta_n)_{xt}\|_0\big)\|\ddot{w}_n\|_0\\\leq& C\big(\|\dot{w}_n\|_{2}+\|\dot{\theta}_n\|_{2}+\|\dot{w}_n\|_{1}+\|\dot{\theta}_n\|_{1}\big)\|\ddot{w}_n\|_0\\\leq& C\big(\|\ddot{w}_n\|^2_0+\|\dot{w}_n\|^2_{2}+\|\dot{\theta}_n\|^2_{2}\big),
	\end{split}
\end{equation*}
where in the last inequality we apply \eqref{spectral}, being $\dot{w}_n, \dot{\theta}_n\in (H^2\cap H_0^1)(I)$, and Young's inequality.
We also compute
$$
|\big[\mathcal{L}(w_n\pm\ell\theta_n)\big]_t|=\bigg|\int_0^\pi \dfrac{(w_n\pm\ell\theta_n+s)_x(w_n\pm\ell\theta_n)_{xt}}{\Xi(w_n\pm\ell\theta_n)}\bigg|\leq C\|(w_n\pm\ell\theta_n)_{xt}\|_0
$$
so that
\begin{equation*}
	\begin{split}
		\bigg|\big[\mathcal{L}(w_n\pm\ell\theta_n)\big]_t\bigg(\dfrac{(w_n\pm\ell\theta_n+s)_{xx}}{\Xi(w_n\pm\ell\theta_n)^3},\ddot{w}_n\bigg)_0\bigg|\leq& C\|(w_n\pm\ell\theta_n)_{xt}\|_0\|\ddot{w}_n\|_0\\\leq& C\big(\|\ddot{w}_n\|^2_0+\|\dot{w}_n\|^2_{2}+\|\dot{\theta}_n\|^2_{2}\big).
	\end{split}
\end{equation*}
Therefore, looking at \eqref{melan42}, we infer
$$
|\big([f(w_n,\theta_n)]_{xt},\ddot{w}_n\big)_0+\big([\overline f(w_n,\theta_n)]_{xt},\ddot{\theta}_n\big)_0|\leq C\big(\|\ddot{w}_n\|^2_0+\|\dot{w}_n\|^2_{2}+\|\ddot{\theta}_n\|^2_{0}+\|\dot{\theta}_n\|^2_{2}\big).
$$

Applying estimates as \eqref{bd11} to the last right hand side term of \eqref{eq}, integrating over $(0,t)$ and repeating similar passages as in \eqref{step 2.2}, we get $C,c>0$ such that
\begin{equation*}\label{eq2}
	\begin{split}
		&\|\ddot{w}_n(t)\|_0^2+\|\dot{w}_n(t)\|^2_{2}+\|\ddot{\theta}_n(t)\|_0^2+\|\dot{\theta}_n(t)\|^2_{2}\\&\leq C\big(\|\ddot{w}_n(0)\|_0^2+\|\dot{w}_n(0)\|^2_{2}+\|\ddot{\theta}_n(0)\|_0^2+\|\dot{\theta}_n(0)\|^2_{2}+\|\dot{w}_n(0)\|^4_{2}+1\big)\\&\hspace{5mm}+ c\int_0^t\big(\|\ddot{w}_n(s)\|^2_0+\|\dot{w}_n(s)\|^2_{2}+\|\ddot{\theta}_n(s)\|^2_{0}+\|\dot{\theta}_n(s)\|^2_{2}\big)\,ds.
	\end{split}	
\end{equation*}
Hence, by Gronwall inequality, we obtain
$$
\|\ddot{w}_n(t)\|_2^0+\|\dot{w}_n(t)\|^2_{2}+\|\ddot{\theta}_n(t)\|_0^2+ \|\dot{\theta}_n(t)\|^2_{2}\leq C\big(\|\ddot{w}_n(0)\|_0^2+\|\dot{w}_n(0)\|^2_{2}+\|\ddot{\theta}_n(0)\|_0^2+\|\dot{\theta}_n(0)\|^2_{2}+\|\dot{w}_n(0)\|^4_{2}+1\big)e^{cT};
$$
since $w^1,\theta^1\in (H^2\cap H_0^1)(I)$, the uniform boundedness of $\|\dot{w}_n(0)\|_{2}$ and $\|\dot{\theta}_n(0)\|_{2}$ is obtained for all $t\in[0,T]$. 
Let us consider \eqref{eq weak sist lin gal  2n} in $t=0$; since 
$$
|\big(f(w_n(0),\theta_n(0)),e'_k\big)_0|\leq C\big(\|w_n(0)\|_{1}+\|\theta_n(0)\|_{1}\big)\leq C\big(\|w_0\|_{1}+\|\theta_0\|_{1}\big)
$$
and, similarly $|\big(\overline f(w_n(0),\theta_n(0)),e'_k\big)_0|$, we infer that also $\|\ddot{w}_n(0)\|_{0}$ and $\|\ddot{\theta}_n(0)\|_{0}$ are uniformly bounded for all $t\in[0,T]$. This implies that
$$
\|\ddot{w}_n(t)\|_0,\,\,\|\ddot{\theta}_n(t)\|_0 \qquad \|\dot{w}_n(t)\|_{2},\,\, \|\dot{\theta}_n(t)\|_{2}\quad \text{are uniformly bounded for all }t\in [0,T].
$$
Then by the equations
\begin{equation*}\label{strong}
	\begin{cases}
		(w_n)_{xxxx}=-(w_n)_{tt}-\mu (w_n)_t-[P-S\|w_n\|^2_1](w_n)_{xx}-[f(w_n,\theta_n)]_x+g-\beta \Upsilon(\theta_n)_t-\eta\theta_n\\
		(\theta_n)_{xxxx}=-\dfrac{1}{\epsilon}\bigg(\dfrac{\ell^2}{3}(\theta_n)_{tt}+\zeta (\theta_n)_t-\kappa(\theta_n)_{xx}+[\overline f(w_n,\theta_n)]_x\bigg),
	\end{cases}
\end{equation*}
we infer that $(w_n)_{xxxx}$ and $(\theta_n)_{xxxx}$ are uniformly bounded in $L^2(I)$ for all $t\in[0,T]$ and then, that $w_n$, $\theta_n$ are uniformly bounded in $H^4(I)$ for all $t\in[0,T]$. The  final regularity of the strong solution, as stated in the theorem, can be obtained arguing as in the proof of existence of weak solution. Recovery of the second order boundary conditions from the weak form \eqref{weak} is standard.

\textit{$\bullet$ Energy identity.} From \eqref{weak} and the uniqueness of a weak solution, we obtain that $(w,\theta)$ is the limit of the sequence $(w_n,\theta_n)$ built in the first step. Thanks to the strong convergence in \eqref{conv}, we can take the limit in \eqref{step 2.1} for each of the linear terms. To the nonlinearity $\Pi(w_n\pm\ell\theta_n)$ we apply \eqref{Pi}, so that it holds
\begin{equation*}
	\begin{split}
		|\Pi(w_n\pm\ell\theta_n)-\Pi(w\pm\ell\theta)|&\leq C\big(\|w_n(t)-w(t)\|_{1}+\|\theta_n(t)-\theta(t)\|_{1}\big)\\&\leq C\big(\|w_n-w\|_{C^0([0,T];H^1(I))}+\|\theta_n-\theta\|_{C^0([0,T];H^1(I))}\big)\rightarrow 0 \quad \text{as }n\rightarrow\infty,
	\end{split}
\end{equation*}
thanks to \eqref{conv}; applying \eqref{S}, we observe the strong convergence of the Woinowsky-Krieger nonlinearity capturing the stretching of the deck. Therefore, it is possible to pass the limit in all the terms of \eqref{step 2.1}, getting
$$\mathcal{E}(t)+\mu\int_0^t\|w_t(\tau)\|_0^2d\tau+\zeta \int_0^t\|\theta_t(\tau)\|_0^2d\tau=\mathcal{E}(0)-\beta \Upsilon\int_0^t\big(\theta_t(\tau),w_t(\tau)\big)_0d\tau-\eta\int_0^t\big(\theta(\tau),w_t(\tau)\big)_0d\tau.$$
The thesis follows considering $s$ instead of 0 as the lower bound of integration in the previous equality, and repeating the arguments similarly.

\subsection{Construction of an absorbing ball: proof of Proposition \ref{prop-absorbing}}\label{proof-abs}
We introduce a Lyapunov-type function depending on $\nu,\mu,\zeta>0$
\begin{equation}\label{V}
\begin{split}
V_{\nu,\mu,\zeta}(S_t(y)):=	&\mathcal{E}(t)+\nu \big(w_t(t),w(t)\big)_0+\dfrac{\nu\mu}{2}\|w(t)\|_0^2+ \nu\big(\theta_t(t),\theta(t)\big)_0+\dfrac{\nu\zeta}{2}\|\theta(t)\|_0^2+\\&\beta \Upsilon\big(\theta_t(t),w(t)\big)_0+\eta \big(\theta(t),w(t)\big)_0,
\end{split}
\end{equation}
where $S_t(y)=\big(w(t),w_t(t),\theta(t),\theta_t(t)\big)$ is as before for $t\geq 0$, and with $\nu$ to be specified later.
First of all we prove that the Lyapunov function is bounded by the positive portion of the energy.
\begin{lemma}\label{lemma0}
	There exists $\overline \nu(\mu,\zeta)>0$ such that if $\nu\in(0,\overline \nu)$, there are $c_0(\mu,\zeta)$, $c_1(\mu,\zeta)$ and\\ $c_2(\mu,\zeta,P,S,\beta,\Upsilon,\eta,\|g\|_0)>0$ so that
	\begin{equation}\label{eqq1}
		c_0E_+(t)-c_2\leq V_{\nu,\mu,\zeta}(S_t(y))\leq c_1E_+(t)+c_2.
	\end{equation}
\end{lemma}
\begin{proof}
	We claim there exists $M(\nu,P,S,\|g\|_0)>0$ such that
	\begin{equation}\label{ts1}
		(1-\nu)E_+-M\leq \mathcal{E}\leq (1+\nu)E_++M\qquad \forall\nu\in(0,1).
	\end{equation}
	Since $\mathcal{E}=E_+-\dfrac{P}{2}\|w\|_1^2-\big(g,w\big)_0$ we apply the H\"older and Young inequalities
	$$
	\dfrac{P}{2}\|w\|_1^2\leq \nu S\dfrac{\|w\|_1^4}{4}+\dfrac{P^2}{4S\nu}\qquad \big(g,w\big)_0\leq \nu\|w\|_0^2+\dfrac{1}{4\nu}\|g\|_0^2\leq \nu\|w\|_2^2+\dfrac{1}{4\nu}\|g\|_0^2,
	$$
	inferring \eqref{ts1} with $M=\frac{P^2}{4S\nu}+\frac{\|g\|_0^2}{4\nu}$.
	Looking at $V_{\nu,\mu,\zeta}$ in \eqref{V}, we need the following bounds for some $\gamma_i>0$
	\begin{equation}\label{bds}
		\begin{split}
			&\nu\big(w_t,w\big)_0\leq \gamma_1\dfrac{\|w_t\|_0^2}{2}+\dfrac{\nu^2}{\gamma_1}\dfrac{\|w\|_0^2}{2},\qquad\qquad  \nu\big(\theta_t,\theta\big)_0\leq \gamma_2\dfrac{\|\theta_t\|_0^2}{2}+\dfrac{\nu^2}{\gamma_2}\dfrac{\|\theta\|_0^2}{2},\\
			&	\beta\Upsilon\big(\theta_t,w\big)_0\leq \gamma_2\dfrac{\|\theta_t\|_0^2}{2}+\dfrac{\beta^2\Upsilon^2}{\gamma_2}\dfrac{\|w\|_0^2}{2}\leq \gamma_2\dfrac{\|\theta_t\|_0^2}{2}+\dfrac{S\nu}{2}\dfrac{\|w\|_1^4}{4}+\dfrac{\beta^4\Upsilon^4}{2S\nu\gamma_2^2},\\
			&	\eta\big(\theta,w\big)_0\leq \gamma_3\dfrac{\|\theta\|_0^2}{2}+\dfrac{\eta^2}{\gamma_3}\dfrac{\|w\|_0^2}{2}\leq \gamma_3\dfrac{\|\theta\|_1^2}{2}+\dfrac{S\nu}{2}\dfrac{\|w\|_1^4}{4}+\dfrac{\eta^4}{2S\nu\gamma_3^2}.
		\end{split}
	\end{equation}
	Therefore, we have
	\begin{equation*}
		\begin{split}
			\nu \big(w_t,&w\big)_0+\dfrac{\nu\mu}{2}\|w\|_0^2+ \nu\big(\theta_t,\theta\big)_0+\dfrac{\nu\zeta}{2}\|\theta\|_0^2+\beta \Upsilon\big(\theta_t,w\big)_0+\eta \big(\theta,w\big)_0\\
			\geq &-\gamma_1\dfrac{\|w_t\|_0^2}{2}-2\gamma_2\dfrac{\|\theta_t\|_0^2}{2}-\gamma_3\dfrac{\|\theta\|_1^2}{2}+\bigg(\nu\mu-\dfrac{\nu^2}{\gamma_1}\bigg)\dfrac{\|w\|^2_0}{2}+\bigg(\nu\zeta-\dfrac{\nu^2}{\gamma_2}\bigg)\dfrac{\|\theta\|^2_0}{2}\\&-S\nu\dfrac{\|w\|_1^4}{4}-\dfrac{\beta^4\Upsilon^4}{2S\nu\gamma_2^2}-\dfrac{\eta^4}{2S\nu\gamma_3^2},
		\end{split}
	\end{equation*}
	hence, taking $\gamma_1=\nu/\mu$, $\gamma_2=\nu/\zeta$ and $\gamma_3=\nu\kappa$, we find from \eqref{ts1}
	\begin{equation*}
		\begin{split}
			V_{\nu,\mu,\zeta}
			\geq& \bigg[1-\nu\bigg(\dfrac{\mu+1}{\mu}\bigg)\bigg]\dfrac{\|w_t\|_0^2}{2}+\bigg[1-\nu\bigg(\dfrac{\zeta+2}{\zeta}\bigg)\bigg]\dfrac{\|\theta_t\|_0^2}{2}+(1-2\nu)\kappa\dfrac{\|\theta\|_1^2}{2}+(1-2\nu)S\dfrac{\|w\|_1^4}{4}\\&-M-\dfrac{\beta^4\Upsilon^4\zeta^2}{2S\nu^3}-\dfrac{\eta^4}{2S\nu^3\kappa^2}\geq c_0E_+-c_2,
		\end{split}
	\end{equation*}
	for all $0<\nu<\min\bigg\{\dfrac{1}{2},\dfrac{\mu}{\mu+1},\dfrac{\zeta}{\zeta+2}\bigg\}:=\overline\nu$. Next, with the same choices of the parameters $\nu,\gamma_i$ and \eqref{bds} we obtain 
	\begin{equation*}
		\begin{split}
			&\nu \big(w_t,w\big)_0+\dfrac{\nu\mu}{2}\|w\|_0^2+ \nu\big(\theta_t,\theta\big)_0+\dfrac{\nu\zeta}{2}\|\theta\|_0^2+\beta \Upsilon\big(\theta_t,w\big)_0+\eta \big(\theta,w\big)_0\\
			&\quad\leq \dfrac{\nu}{\mu}\dfrac{\|w_t\|_0^2}{2}+2\frac{\nu}{\zeta}\dfrac{\|\theta_t\|_0^2}{2}+\nu\kappa\dfrac{\|\theta\|_1^2}{2}+2\nu\mu\dfrac{\|w\|^2_0}{2}+2\nu\zeta\dfrac{\|\theta\|^2_0}{2}+S\nu\dfrac{\|w\|_1^4}{4}+\dfrac{\beta^4\Upsilon^4\zeta^2}{2S\nu^3}+\dfrac{\eta^4}{2S\nu^3\kappa^2}
		\end{split}
	\end{equation*}
	and, in turn, the thesis \eqref{eqq1}.
\end{proof}

We want a bound on the derivative of the Lyapunov function introduced in \eqref{V}. To do this we need a preliminary bound on the $L^2(I)$ norm of the nonlinear terms related to the cables and a lemma providing control of lower frequencies in the $w$ dynamics.

The following lemma is easily adapted from \cite{bongazlasweb,holawe}, where the proof is provided; also see \cite{HHWW}.
\begin{lemma}\cite[Lemma 4.8]{bongazlasweb}\label{eps-lemma}
	For any $s\in(0,2]$ and $\gamma>0$ there exists $C_{\gamma,s}>0$ such that
	\begin{equation*}\label{eps-ineq}
		\|w\|^2_{2-s}\leq \gamma\big(\|w\|_2^2+\|w\|_1^4\big)+C_{\gamma,s}\qquad \forall w\in H^2\cap H_0^1(I).
	\end{equation*}
\end{lemma}
We are ready now to prove a bound on the derivative of $V_{\nu,\mu,\zeta}(S_t(y))$ introduced in \eqref{V}.
\begin{lemma}\label{lemma3}
	Let $V_{\nu,\mu,\zeta}$ be as in \eqref{V}. For all $\mu,\zeta>0$ there exist $\overline \nu(\mu,\zeta)$, $\overline\epsilon(\nu, \beta,\ell)>0$ such that if $\nu\in(0,\overline\nu)$ and $\epsilon\in(0,\overline\epsilon)$, then there are $c_3(\mu,\zeta,\nu,\kappa,\epsilon,P,S,\beta,\Upsilon,\eta,\ell,c_c)$, $c_4(\nu,\ell,\|g\|_0,C_c,\overline C_c,\overline c_c)>0$ so that
	\begin{equation}\label{ineq3}
		\dfrac{d}{dt} V_{\nu,\mu,\zeta}(S_t(y))\leq -c_3E_+(t)+c_4.
	\end{equation}
\end{lemma}
\begin{proof}
	We suppose that $y(t)=\big(w(t),w_t(t),\theta(t),\theta_t(t)\big)$ is a smooth solution of \eqref{eq_sist1} (we can extend by density to weak solutions as the final step). Then we compute the time derivative of $V_{\nu,\mu,\zeta}$, i.e.
	\begin{equation*}
		\begin{split}
			\dfrac{d}{dt} V_{\nu,\mu,\zeta}&=\dfrac{d}{dt}\mathcal{E}+\nu (w_{tt},w)_0+\nu\|w_t\|_0^2+\nu\mu(w,w_t)_0+ \nu(\theta_{tt},\theta)_0+\nu\|\theta_t\|_0^2+\nu\zeta(\theta,\theta_t)_0\\&+\beta \Upsilon\dfrac{d}{dt}\big(\theta_t,w\big)_0+\eta\dfrac{d}{dt} \big(\theta,w\big)_0.
		\end{split}
	\end{equation*} 
	From \eqref{energy} we infer
	$$
	\dfrac{d}{dt}\mathcal{E}=-\mu\|w_t\|_0^2-\zeta \|\theta_t\|_0^2-\beta \Upsilon(\theta_t,w_t)_0-\eta(\theta,w_t)_0,
	$$
	while testing the equations in \eqref{weak} respectively by $w$ and $\theta$ we find 
	\begin{equation*}
		\begin{split}
			\dfrac{d}{dt} V_{\nu,\mu,\zeta}&=(\nu-\mu)\|w_t\|_0^2+(\nu-\zeta) \|\theta_t\|_0^2+P\nu\|w\|^2_1-\nu\|w\|^2_{2}- \nu\epsilon\|\theta\|^2_{2}-\nu\kappa\|\theta\|^2_{1}+\\&-\beta \Upsilon(\theta_t,w_t)_0-\eta(\theta,w_t)_0-\beta \Upsilon\nu(\theta_t,w)_0-\nu\eta(\theta,w)_0+ \nu(g,w)_0\\&-S\nu\|w\|^4_1+\nu(f(w,\theta),w_x)_0+\nu(\overline f(w,\theta),\theta_x)_0
			+\beta \Upsilon\dfrac{d}{dt}\big(\theta_t,w\big)_0+\eta\dfrac{d}{dt} \big(\theta,w\big)_0.
		\end{split}
	\end{equation*}
	We rewrite the following terms using the product rule in time as
	\begin{equation*}
		\begin{split}
			(\theta,w_t)_0&=\dfrac{d}{dt} \big(\theta,w\big)_0-\big(\theta_t,w\big)_0\\
			(\theta_t,w_t)_0&=\dfrac{d}{dt}\big(\theta_t,w\big)_0-\big(\theta_{tt},w\big)_0\\&=\dfrac{d}{dt}\big(\theta_t,w\big)_0+\dfrac{3\zeta}{\ell^2}\big(\theta_t,w\big)_0 +\dfrac{3\epsilon}{\ell^2}\big(\theta,w\big)_2+\dfrac{3\kappa}{\ell^2}\big(\theta,w)_1-\dfrac{3}{\ell^2}\big(\overline{ f}(w,\theta),w_x\big)_0,
		\end{split}
	\end{equation*}
	where we used \eqref{eq_sist1}$_2$. Hence, we obtain
	\begin{equation}\label{ineq}
		\begin{split}
			\dfrac{d}{dt} V_{\nu,\mu,\zeta}&=(\nu-\mu)\|w_t\|_0^2+(\nu-\zeta) \|\theta_t\|_0^2+P\nu\|w\|^2_1-\nu\|w\|^2_{2}-\nu\epsilon\|\theta\|^2_{2}-\nu\kappa\|\theta\|^2_{1}-S\nu\|w\|^4_1\\&+(\eta-\beta \Upsilon\zeta\tfrac{3}{\ell^2}-\beta \Upsilon\nu)\big(\theta_t,w\big)_0 -\beta \Upsilon\epsilon\tfrac{3}{\ell^2}\big(\theta,w\big)_2-\beta \Upsilon\kappa\tfrac{3}{\ell^2}\big(\theta,w)_1-\nu\eta(\theta,w)_0+ \nu(g,w)_0\\&+\nu(f(w,\theta),w_x)_0+\nu(\overline f(w,\theta),\theta_x)_0+\beta \Upsilon\tfrac{3}{\ell^2}\big(\overline{ f}(w,\theta),w_x\big)_0.
		\end{split}
	\end{equation}
	
	Next we bound the right hand side of \eqref{ineq}. 
	First of all we apply the H\"older and Young inequalities ($\eps_1,\eps>0$) 
	\begin{equation*}
		\begin{split}
			&|(\eta-\beta \Upsilon\zeta-\beta \Upsilon\nu)\big(\theta_t,w\big)_0|\leq \zeta\dfrac{\|\theta_t\|^2_0}{2}+(\eta-\beta \Upsilon\zeta-\beta \Upsilon\nu)^2\dfrac{\|w\|^2_0}{2\zeta}\\&|\beta \Upsilon\epsilon\tfrac{3}{\ell^2}\big(\theta,w\big)_2|\leq \eps_1\dfrac{\|\theta\|^2_2}{2}+\dfrac{9\beta^2\Upsilon^2\epsilon^2}{\eps_1\ell^4}\dfrac{\|w\|_2^2}{2}\qquad|\beta \Upsilon\kappa\tfrac{3}{\ell^2}\big(\theta,w)_1|\leq \eps\dfrac{\|\theta\|^2_1}{2}+\dfrac{9\beta^2\Upsilon^2\kappa^2}{\eps\ell^4}\dfrac{\|w\|_1^2}{2}\\ &|\nu\eta(\theta,w)_0|\leq \eps\dfrac{\|\theta\|^2_0}{2}+\dfrac{\nu^2\eta^2}{\eps}\dfrac{\|w\|_0^2}{2}\qquad\hspace{15mm}
			|\nu(g,w)_0|\leq \eps\dfrac{\|w\|^2_0}{2}+\dfrac{\nu^2}{\eps}\frac{\|g\|^2_0}{2}.
		\end{split}
	\end{equation*}
	%%About the cable nonlinearity, recalling \eqref{melan42}, we apply H\"older, Young inequalities ($\eps>0$) and \eqref{h-L2} we obtain
	\begin{equation*}
		\begin{split}
			|\beta \Upsilon\tfrac{3}{\ell^2}(\overline f(w,\theta),w_x)_0|=&|\beta \Upsilon\tfrac{3}{\ell}(h(w+\ell\theta)-h(w-\ell\theta),w_x)_0|\\\leq& \dfrac{\eps}{2}\|h(w+\ell\theta)-h(w-\ell\theta)\|_0^2+\beta^2\Upsilon^2\frac{9}{\ell^2}\dfrac{\|w\|^2_1}{2\eps}\\\leq& c_c\eps\big(\|w_x+\ell\theta_x\|^2_{L^1(I)}+\|w_x-\ell\theta_x\|^2_{L^1(I)}\big)+\beta^2\Upsilon^2\frac{9}{\ell^2}\dfrac{\|w\|^2_1}{2\eps}+C\\\leq& 4\pi c_c\eps\big(\|w\|_1^2+\ell^2\|\theta\|^2_{1}\big)+\beta^2\Upsilon^2\frac{9}{\ell^2}\dfrac{\|w\|^2_1}{2\eps}+C,
		\end{split}
	\end{equation*}
	where we apply \eqref{h-L2}.
	The remaining cable nonlinear terms can be combined, recalling \eqref{melan42}; through \eqref{h-weak} we find
	\begin{equation*}
		\begin{split}
			(f(w,\theta),w_x)_0+&(\overline f(w,\theta),\theta_x)_0=\big(h(w+\ell\theta),w_x+\ell\theta_x\big)_0+\big(h(w-\ell\theta),w_x-\ell\theta_x\big)_0\\
			\leq &-\Pi(w+\ell\theta)-\Pi(w-\ell\theta)+C_c\|w_x+\ell\theta_x\|_{L^1(I)}+C_c\|w_x-\ell\theta_x\|_{L^1(I)}+2\overline C_c\\
			\leq &-\Pi(w+\ell\theta)-\Pi(w-\ell\theta)+\eps\dfrac{\|w\|^2_1}{2\nu} +\eps\dfrac{\|\theta\|^2_1}{2\nu}+C,
		\end{split}
	\end{equation*} 
	so that \eqref{ineq} becomes
	\begin{equation*}\label{ineq0}
		\begin{split}
			\dfrac{d}{dt} V_{\nu,\mu,\zeta}\leq&~(\nu-\mu)\|w_t\|_0^2+(2 \nu-\zeta) \dfrac{\|\theta_t\|_0^2}{2}-S\nu\|w\|^4_1\\
			&+\bigg(\dfrac{(\eta-\beta \Upsilon\zeta-\beta \Upsilon\nu)^2}{\zeta}+\frac{\nu^2\eta^2}{\eps}+\eps\bigg)\dfrac{\|w\|_0^2}{2}+\bigg(2P\nu+\dfrac{9\beta^2\Upsilon^2\kappa^2}{\eps\ell^4}+\dfrac{9\beta^2\Upsilon^2}{\eps\ell^2}\bigg)\dfrac{\|w\|^2_1}{2}\\&+\bigg(\dfrac{9\beta^2\Upsilon^2\epsilon^2}{\eps_1\ell^4}-2\nu\bigg)\frac{\|w\|^2_{2}}{2}+\eps\dfrac{\|\theta\|^2_0}{2}+(\eps-2\nu\kappa)\dfrac{\|\theta\|^2_1}{2}+(\eps_1-2 \nu\epsilon)\frac{\|\theta\|^2_{2}}{2}+\dfrac{\nu^2}{\eps}\frac{\|g\|^2_0}{2}\\&-\nu\Pi(w+\ell\theta)-\nu\Pi(w-\ell\theta)+\eps\dfrac{\|w\|^2_1}{2} +\eps\dfrac{\|\theta\|^2_1}{2}\\&+4\pi c_c\eps\big(\|w\|_1^2+\ell^2\|\theta\|^2_{1}\big)+\beta^2 \Upsilon^2\dfrac{9}{\ell^2}\dfrac{\|w\|^2_1}{2\eps}+C.
		\end{split}
	\end{equation*} 
	Using \eqref{spectral} and collecting the terms we find
	\begin{equation}\label{ineq4}
		\begin{split}
			\dfrac{d}{dt} V_{\nu,\mu,\zeta}&\leq(\nu-\mu)\|w_t\|_0^2+(2\nu-\zeta) \dfrac{\|\theta_t\|_0^2}{2}+\big(\eps(3+8\pi c_c\ell^2)-2\nu\kappa\big)\frac{\|\theta\|^2_{1}}{2}+(\eps_1-2\nu\epsilon)\frac{\|\theta\|^2_{2}}{2}\\&-\nu\Pi(w+\ell\theta)-\nu\Pi(w-\ell\theta)-S\nu\|w\|^4_1+\bigg(\dfrac{9\beta^2\Upsilon^2\epsilon^2}{\eps_1\ell^4}-2\nu\bigg)\frac{\|w\|^2_{2}}{2}\\
			&+\bigg(2\eps+8\pi c_c\eps+2P\nu+\dfrac{(\eta-\beta \Upsilon\zeta-\beta \Upsilon\nu)^2}{\zeta}+\frac{\nu^2\eta^2}{\eps}+\dfrac{9\beta^2\Upsilon^2\kappa^2}{\eps\ell^4}+\dfrac{9\beta^2 \Upsilon^2}{\eps\ell^2} \bigg)\frac{\|w\|^2_{1}}{2} +C.
		\end{split}
	\end{equation}
	To guarantee the negativity of the terms on the first and second lines of right hand side of \eqref{ineq4} we choose $\eps_1= \frac{3}{2}\nu\epsilon$ and
	\begin{equation}\label{parameters}
		0<\nu<\min\{\mu,\zeta/2\}:=\overline\nu, \qquad0<\eps<\frac{2\nu\kappa}{3+8\pi c_c\ell^2}\qquad 0<\epsilon<\frac{ \nu^2\ell^2}{3\beta^2}:=\overline \epsilon.
	\end{equation}
	From Lemma \ref{eps-lemma} we infer the existence of $\gamma>0$ and $C_\gamma>0$ such that
	$$
	\|w\|^2_{1}\leq \gamma\big(\|w\|_2^2+\|w\|_1^4\big)+C_{\gamma}\qquad \forall w\in H^2\cap H_0^1(I),
	$$
	yielding to
	\begin{equation*}\label{ineq2}
		\begin{split}
			\dfrac{d}{dt} V_{\nu,\mu,\overline\nu,\zeta}&\leq(\nu-\mu)\|w_t\|_0^2+(2\nu-\zeta) \dfrac{\|\theta_t\|_0^2}{2}-\nu\Pi(w+\ell\theta)-\nu\Pi(w-\ell\theta)\\&+\bigg[\bigg(2\eps+8\pi c_c\eps+2P\nu+\dfrac{(\eta-\beta \Upsilon\zeta-\beta \Upsilon\nu)^2}{\zeta}+\frac{\nu^2\eta^2}{\eps}+\dfrac{9\beta^2\Upsilon^2\kappa^2}{\eps\ell^4}+\dfrac{9\beta^2 \Upsilon^2}{\eps\ell^2} \bigg)\gamma-2S\nu\bigg]\dfrac{\|w\|^4_1}{2}\\
			&+\bigg[\bigg(2\eps+8\pi c_c\eps+2P\nu+\dfrac{(\eta-\beta \Upsilon\zeta-\beta \Upsilon\nu)^2}{\zeta}+\frac{\nu^2\eta^2}{\eps}+\dfrac{9\beta^2\Upsilon^2\kappa^2}{\eps\ell^4}+\dfrac{9\beta^2 \Upsilon^2}{\eps\ell^2} \bigg)\gamma \\&-2\bigg(\nu-\dfrac{3\beta^2\Upsilon^2\epsilon}{\nu\ell^4}\bigg)\bigg]\frac{\|w\|^2_{2}}{2}+\big(\eps(3+8\pi c_c\ell^2)-2\nu\kappa\big)\frac{\|\theta\|^2_{1}}{2}-\frac{\nu\epsilon}{2}\frac{\|\theta\|^2_{2}}{2}+c_4.
		\end{split}
	\end{equation*} 
	Then we take the parameters in \eqref{parameters} and
	\begin{equation*}\label{constants}
		\begin{split}
			\gamma&<\dfrac{2\min\bigg\{S\nu\,,\,\nu-\dfrac{3\beta^2\Upsilon^2\epsilon}{\nu\ell^4}\bigg\}}{2\eps+8\pi c_c\eps+2P\nu+\dfrac{(\eta-\beta \Upsilon\zeta-\beta \Upsilon\nu)^2}{\zeta}+\dfrac{\nu^2\eta^2}{\eps}+\dfrac{9\beta^2\Upsilon^2\kappa^2}{\eps\ell^4}+\dfrac{9\beta^2 \Upsilon^2}{\eps\ell^2}},
		\end{split}
	\end{equation*} 
	implying \eqref{ineq3}.
\end{proof}
We are now in position to complete the proof of the Proposition \ref{prop-absorbing}. From Lemma \ref{lemma0} and Lemma \ref{lemma3} we have for some $\Lambda(\nu)>0$ and $C>0$ that
\begin{equation*}\label{bd2}
	\dfrac{d}{dt}V_{\nu,\mu,\zeta}(S_t(y))+\Lambda V_{\nu,\mu,\zeta}(S_t(y))\leq C,\quad t>0;
\end{equation*}
integrating, this implies
\begin{equation*}
	V_{\nu,\mu,\zeta}(S_t(y))\leq V_{\nu,\mu,\zeta}(y)e^{-\Lambda t}+\dfrac{C}{\Lambda}(1-e^{-\Lambda t}).
\end{equation*}
Therefore, the set 
$$
\mathcal{B}:=\Big\{z\in Y:V_{\nu,\mu,\zeta}(z)\leq 1+\dfrac{C}{\Lambda}\Big\}
$$
is a bounded, forward-invariant absorbing set, implying that $(S_t, Y)$ is ultimately dissipative (in the sense of Section \ref{app1}).

\subsection{Quasi-stability and attractors: proof of Theorem \ref{teo-main}}\label{proof-teo-main}
We construct here the global compact attractor for the dynamical system \eqref{eq_sist1} using quasi-stability theory, e.g. see \cite{chueshov}. A quasi-stable dynamical system is one where the difference of two trajectories can be decomposed into uniformly stable and compact parts; in this way it is also possible to obtain, almost immediately, that the attractor is smooth, with finite fractal dimension and that there exists a generalized fractal exponential attractor. We follow the program outlined in \cite{holawe}, based on \cite{chla} and, later \cite{chueshov}.

Let $\mathfrak{g}(w):=\big(P-S\int_Iw_{x}^2\big)w_{xx}$ and $f, \overline f$ the usual cable nonlinearity in \eqref{melan42}.
We consider the difference of two strong solutions $(w^i,\theta^i)$, $i=1,2$ to \eqref{eq_sist1}, satisfying
\begin{equation}
	\left\{\begin{array}{ll}
		W_{tt}+\mu\, W_t+W_{xxxx}+\mathfrak{g}(w^\mathcal{I})-\mathfrak{g}(w^\mathcal{II})+\big(f(w^\mathcal{I},\theta^\mathcal{I})-f(w^\mathcal{II},\theta^\mathcal{II})\big)_x=-\beta \Upsilon\Theta_t-\eta\Theta &\text{in }I_T\\
		\frac{\ell^2}{3}\Theta_{tt}+\zeta\, \Theta_t +\epsilon\Theta_{xxxx}-\kappa\Theta_{xx}+\big(\overline f(w^\mathcal{I},\theta^\mathcal{I})-\overline f(w^\mathcal{II},\theta^\mathcal{II})\big)_x=0\,& \text{in }I_T\\
		W=W_{xx}=\Theta=\Theta_{xx}=0 &\hspace{-20mm}\text{on }\{0,\pi\}\times(0,T)\\
		W(x,0)=w^{\mathcal{I}}_0(x)-w^{\mathcal{II}}_0(x),\quad\,\Theta(x,0)=\theta^\mathcal{I}_0(x)-\theta^\mathcal{II}_0(x) &\text{on }\overline I\\
		W_t(x,0)=w^\mathcal{I}_1(x)-w^\mathcal{II}_1(x),\quad\Theta_t(x,0)=\theta^\mathcal{I}_1(x)-\theta^\mathcal{II}_1(x) &\text{on }\overline I
	\end{array}\right.
	\label{eq_sist_diff}
\end{equation}
where $W=w^\mathcal{I}-w^\mathcal{II}$ and $\Theta=\theta^\mathcal{I}-\theta^\mathcal{II}$; 
we recall that, on any bounded, forward-invariant ball $B_R(Y)$ ($R>0$ is the radius), we have
$$
\|w_t^\mathcal{I}(t)\|_0+\|w^\mathcal{I}(t)\|_2+\|\theta_t^\mathcal{I}(t)\|_0+\|\theta^\mathcal{I}(t)\|_2+\|w_t^\mathcal{II}(t)\|_0+\|w^\mathcal{II}(t)\|_2+\|\theta_t^\mathcal{II}(t)\|_0+\|\theta^\mathcal{II}(t)\|_2\leq C(R),\quad t>0.
$$

We introduce 
\begin{equation*}
	\mathcal{G}(W):=\mathfrak{g}(w^\mathcal{I})-\mathfrak{g}(w^\mathcal{II})\qquad \mathcal{F}(W,\Theta):=f(w^\mathcal{I},\theta^\mathcal{I})-f(w^\mathcal{II},\theta^\mathcal{II})\qquad \mathcal{\overline F}(W,\Theta):=\overline f(w^\mathcal{I},\theta^\mathcal{I})-\overline f(w^\mathcal{II},\theta^\mathcal{II})
\end{equation*}
and the ``difference" energy
$$
E_{W,\Theta}(t):=\dfrac{\|W_t\|_0^2}{2}+\dfrac{\|W\|_{2}^2}{2}+\ell^2\dfrac{\|\Theta_t\|^2_{0}}{6}+\epsilon\dfrac{\|\Theta\|^2_{2}}{2}+\kappa\dfrac{\|\Theta\|^2_{1}}{2}.
$$
We associate to \eqref{eq_sist_diff} the following energy identity
\begin{equation}
	\begin{split}\label{en_id}
		E_{W,\Theta}(t)+\mu\int_s^t\|W_t\|_0^2+\zeta \int_s^t\|\Theta_t\|_0^2=&~	E_{W,\Theta}(0)-\beta \Upsilon\int_s^t\big(\Theta_t,W_t\big)_0-\eta\int_s^t\big(\Theta ,W_t\big)_0\\&\hspace{-20mm}-\int_s^t\big( \mathcal{G}(W),W_{t}\big)_0+\int_s^t \,_{-1}\langle \mathcal{F}(W,\Theta),W_{xt}\rangle_{1}+\int_s^t\, _{-1}\langle \mathcal{\overline F}(W,\Theta),\Theta_{xt}\rangle_{1}
	\end{split}
\end{equation}
The following lemma is a special case of \cite[Lemma 8.3.1]{chla}, using \eqref{en_id} and standard ``wave-type" multipliers. It also uses the fact that $\mathfrak{g}, f,\overline f\in \text{Lip}_{\text{loc}}(H^2\cap H^1_0(I), L^2(I))$.
\begin{lemma}\label{lemma_E}
	Let $w^i,\theta^i\in C^0(0,T; (H^2\cap H_0^1)(I))\cap C^1(0,T;L^2(I))$ solve \eqref{eq_sist1} for $i=\mathcal{I},\mathcal{II}$. Additionally assume $(w^i(t), w_t^i(t))$, $(\theta^i(t),\theta_t^i(t))\in B_R(Y)$ for all $t\in [0,T]$ with $T>0$. Then, for any $\eta\in(0,2]$ it holds
	\begin{equation}\label{en3}
		\begin{split}
			TE_{W,\Theta}(T)+\int_0^T&E_{W,\Theta}(\tau)d\tau\leq a_0E_{W,\Theta}(0)+C(R,T,\eta)\sup\limits_{\tau\in[0,T]}\big(\|W\|_{2-\eta}^2+\|\Theta\|_{2-\eta}^2\big)+\\&-a_1\int_0^T\int_s^T\big( \mathcal{G}(W),W_{t}\big)_0d\tau ds-a_2\int_0^T\big( \mathcal{G}(W),W_{t}\big)_0 ds\\&+a_3\int_0^T\int_s^T\, _{-1}\langle \mathcal{F}(W,\Theta),W_{xt}\rangle_{1} d\tau ds+a_4\int_0^T\, _{-1}\langle \mathcal{F}(W,\Theta),W_{xt}\rangle_{1} ds\\&+a_5\int_0^T\int_s^T\, _{-1}\langle \mathcal{\overline F}(W,\Theta),\Theta_{xt}\rangle_{1} d\tau ds+a_6\int_0^T\, _{-1}\langle \mathcal{\overline F}(W,\Theta),\Theta_{xt}\rangle_{1} ds,
		\end{split}
	\end{equation}
with $a_i>0$ not dependent on $T$ and $R$.
\end{lemma}
Now we bound the nonlinear differences term in \eqref{en3}. Concerning the deck nonlinearity, $\mathfrak{g}(w)$, we may estimate in the standard way for Krieger-Woinowsky (or Berger-type) nonlinearities. In \cite{holawe} it is established that  there exist $\eps>0$ and $C(\eps,R)>0$ such that
	\begin{equation}\label{ts3}
	\begin{split}
		&\bigg|\int_s^t\big( \mathcal{G}(W),W_{t}\big)_0 d\tau\bigg|\leq \eps\int_s^tE_{W,\Theta}(\tau)d\tau+C(R,\eps)\sup_{\tau\in[s,t]}\|W(\tau)\|_{2-\eta}^2\qquad \forall \eta\in(0,\tfrac{1}{2}),
	\end{split}	
\end{equation}
provided $w^i\in\mathcal{B}_R(H^2\cap H^1_0(I))$ for all $\tau\in[s,t]$. We do not replicate that proof here, however, we need to produce a similar inequality for the cable nonlinearity. We do so in the next lemma; the proof relies on two highly non-trivial computational lemma---based on a novel decomposition of the cable-hanger nonlinearity---given in the Appendix \ref{app2}.
\begin{lemma}
	Let $W=w^\mathcal{I}-w^\mathcal{II}$, $\Theta=\theta^\mathcal{I}-\theta^\mathcal{II}$, let $f(w,\theta)$ and $\overline f(w,\theta)$ the cable nonlinearity be as in \eqref{melan42}. Also assume $w^i,\theta^i\in C^0([s,t], (H^2\cap H_0^1)(I))\cap C^1([s,t],L^2(I))$ for $i=\mathcal{I},\mathcal{II}$. Then there exist $\eps>0$ and $C(\eps,R)>0$ such that
	\begin{equation}\label{ts2}
		\begin{split}
			&\bigg|\int_s^t\,_{-1}\langle\mathcal{F}(W,\Theta),W_{xt}\rangle_{1} d\tau+\int_s^t\,_{-1}\langle\mathcal{\overline F}(W,\Theta),\Theta_{xt}\rangle_{1} d\tau\bigg|\\&\leq \eps\int_s^tE_{W,\Theta}(\tau)d\tau+C(R,\eps)\sup_{\tau\in[s,t]}(\|W(\tau)\|_{2-\eta}^2+\|\Theta(\tau)\|_{2-\eta}^2),\quad \forall\eta\in(0,\tfrac{1}{2}),
		\end{split}	
	\end{equation}
 provided that $w^i,\theta^i\in\mathcal{B}_R(H^2\cap H^1_0(I))$ for all $\tau\in[s,t]$.
\end{lemma}
\begin{proof}
	Recalling the definition of the nonlinearity in \eqref{melan42} and applying Lemma \ref{zcable} in Appendix \ref{app2} we have
	\begin{equation*}
		\begin{split}
			&\bigg|\int_s^t\,_{-1}\langle f(w^\mathcal{I},\theta^\mathcal{I})-f(w^\mathcal{II},\theta^\mathcal{II}),W_{xt}\rangle_{1} d\tau+\int_s^t\,_{-1}\langle \overline f(w^\mathcal{I},\theta^\mathcal{I})-\overline f(w^\mathcal{II},\theta^\mathcal{II}),\Theta_{xt}\rangle_{1} d\tau\bigg|\\
			=&\bigg|\!\int_s^t\,_{-1}\langle h(w^\mathcal{I}+\ell \theta^\mathcal{I})-h(w^\mathcal{II}+\ell\theta^\mathcal{II}),W_{xt}+\ell\Theta_{xt}\rangle_{1}+\,_{-1}\langle h(w^\mathcal{I}-\ell \theta^\mathcal{I})-h(w^\mathcal{II}-\ell\theta^\mathcal{II}),W_{xt}-\ell\Theta_{xt}\rangle_{1} d\tau\bigg|\\
			\leq&~\eps\int_s^t\bigg[\dfrac{\|W_{t}+\ell\Theta_{t}\|_0^2}{2}+\dfrac{\|W_{t}-\ell\Theta_{t}\|_0^2}{2}+\dfrac{\|W+\ell\Theta\|_1^2}{2}+\dfrac{\|W-\ell\Theta\|_1^2}{2}+\dfrac{\|W+\ell\Theta\|_2^2}{2}+\dfrac{\|W-\ell\Theta\|_2^2}{2}\bigg]d\tau\\&+C(R,\eps)\sup_{\tau\in[s,t]}\big(\|W+\ell\Theta\|^2_{2-\eta}+\|W-\ell\Theta\|^2_{2-\eta}\big).
		\end{split}
	\end{equation*}
	Up to modify the constants $\eps$, $C(R,\eps)>0$ and  using Young inequality, we infer \eqref{ts2}. In particular, this bound holds on the invariant, absorbing ball $\mathcal{B}$ from Proposition \ref{prop-absorbing}.
\end{proof}
We prove the quasi-stability estimate on any bounded, forward-invariant set.
\begin{lemma}\label{quasi-stab}
	Under the assumptions of Proposition \ref{prop-absorbing}, the dynamical system $(S_t,Y)$ corresponding to generalized solutions to \eqref{eq_sist1} is quasi-stable on any bounded, forward-invariant set. In particular, $(S_t,Y)$ is quasi-stable on the absorbing ball $\mathcal{B}$ given in Section \ref{proof-abs}.
\end{lemma}
\begin{proof}
	From \eqref{en3}, \eqref{ts3} and \eqref{ts2}, taking $T$ sufficiently large, we infer that
	\begin{equation*}
		E_{W,\Theta}(T)\leq cE_{W,\Theta}(0)+C(R,T,\eta)\sup_{\tau\in[0,T]}(\|W(\tau)\|_{2-\eta}^2+\|\Theta(\tau)\|_{2-\eta}^2)
	\end{equation*}
with $c<1$ for all $\eta\in(0,1/2)$. By iteration, via the semigroup property, we obtain that there exists $\sigma>0$ such that
\begin{equation*}
	\begin{split}
	\|(W(t),W_t(t);\Theta(t),\Theta_t(t))\|_Y^2&\leq C(\sigma,R)e^{-\sigma t}\|(W(0),W_t(0);\Theta(0),\Theta_t(0))\|_Y^2\\&+C(R,\eta)\sup_{\tau\in[0,t]}(\|W(\tau)\|_{2-\eta}^2+\|\Theta(\tau)\|_{2-\eta}^2).
	\end{split}
\end{equation*}
This implies that $(S_t,Y)$ is quasi-stable on $B_R(Y)$, as defined in Section \ref{app1}.
\end{proof}

To conclude the proof of Theorem \ref{teo-main}, we first apply Corollary \ref{doy*} from Appendix \ref{app1}. This provides the first two bullet points; to obtain the existence of the exponential attractor, we estimate $W(t,s)=w(t)-w(s)$, $\Theta(t,s)=\theta(t)-\theta(s)$ in $\widetilde Y$, using the existence of the absorbing ball $\mathcal B$ and the reasoning exactly as in \cite[Section 5.4]{holawe} and in \cite[Section 5]{bongazlasweb}.

\subsection{Proof of Proposition \ref{lin}}\label{proof-lin}
	From \eqref{odes}$_2$ we obtain \eqref{theta}$_2$, hence, given $\theta_j(t)$ it is possible to compute $w_j(t)$. Under the assumptions on $\mu$ we find
	\begin{equation*}\label{wj}
		w_j(t)=e^{-\frac{\mu}{2}t}\bigg[c^j_1\sin\bigg(\frac{\omega_j}{2}t\bigg)+c^j_2\cos\bigg(\frac{\omega_j}{2} t\bigg)\bigg]+g\dfrac{L^4\sqrt{2L}(1-(-1)^j)}{j^5\pi^5}+w_j^p(t)\quad c_1^j,c_2^j\in\mathbb{R},\,\,\, j\in\mathbb{N}_+,
	\end{equation*}
	where $w^p_j(t)$ is a particular solution similar to $\theta_j(t)$. The condition causing resonance is
	\begin{equation*}\label{res}
		\begin{cases}
			\zeta=\frac{\ell^2}{3}\mu\\
			\omega_j=\frac{3}{\ell^2}\gamma_j,
		\end{cases}
	\end{equation*}
	not occurring from the assumptions.
	Therefore we get $w^p_j(t)=e^{-\frac{3\zeta}{2\ell^2}t}\big[A_j\sin\big(\frac{3}{2\ell^2}\gamma_j t\big)+B_j\cos\big(\frac{3}{2\ell^2}\gamma_j t\big)\big]$ with 
	\begin{equation}\label{A-B}
		\begin{split}
			A_j=&\frac{2 \ell^2 }{3\gamma_j[(4 \pi ^4 j^4 \ell^4/L^4-9 \gamma_j^2-  6 \ell^2\zeta \mu +9 \zeta^2 )^2+36 \gamma_j^2  \left(\ell^2 \mu -3 \zeta \right)^2]}\\& \big\{\left[4 \pi ^4 j^4 \ell^4/L^4-9 \gamma _j^2-6 \ell^2 \zeta\mu +9 \zeta^2\right] \left[9\theta_j^0(\gamma_j^2+\zeta^2)\beta\Upsilon+6\ell^2\theta^1_j\zeta\beta\Upsilon-2\ell^2(3\zeta\theta_j^0+2\ell^2\theta^1_j) \eta \right]+\\& 18\gamma_j^2 \left(3 \zeta -\ell^2 \mu \right) \left[(3\zeta\theta_j^0+2\ell^2\theta^1_j)\beta\Upsilon-3\theta^0_j \zeta\beta\Upsilon+2\ell^2\theta^0_j\eta \right]\big\}\\
			B_j=&\frac{2 \ell^2 }{(4 \pi ^4 j^4 \ell^4/L^4-9 \gamma_j^2-  6 \ell^2\zeta \mu +9 \zeta^2 )^2+36 \gamma_j^2  \left(\ell^2 \mu -3 \zeta \right)^2}\\& \big\{\left[4 \pi ^4 j^4 \ell^4/L^4-9 \gamma _j^2-6 \ell^2 \zeta\mu +9 \zeta^2\right] \left[-(2\ell^2\theta^1_j+3\zeta\theta_j^0) \beta  \Upsilon+3   \theta_j^0 \zeta \beta \Upsilon-2\ell^2 \theta_j^0 \eta  \right]+\\& \left(3 \zeta -\ell^2 \mu \right) \left[18\theta^0_j(\gamma_j^2+\zeta^2)\beta\Upsilon +12\ell^2\theta^1_j\zeta\beta\Upsilon-4\ell^2(3\zeta\theta^0_j+2\ell^2\theta^1_j)\eta \right]\big\}
		\end{split}
	\end{equation}
	Through the initial conditions on $w$ we find
	\begin{equation}\label{c1-c2}
		\begin{cases}
			c_1^j=\frac{1}{\omega_j\ell^2}\bigg(2w^1_j\ell^2+w^0_j\mu\ell^2-3A_j\gamma_j+B_j(3\zeta-\mu\ell^2)-g\tfrac{L^4\sqrt{2L}(1-(-1)^j)}{j^5\pi^5}\mu\ell^2\bigg)\\
			c_2^j=w^0_j-Bj-g\tfrac{L^4\sqrt{2L}(1-(-1)^j)}{j^5\pi^5}.
		\end{cases}
	\end{equation}
	We show that the series \eqref{solution}$_2$ converges uniformly in $I$ for all $t\geq0$. We have
	\begin{equation*}
		\begin{split}
			\bigg|\sum_{j=1}^{\infty} \theta_j(t) \hspace{1mm}\sin\bigg(\dfrac{j\pi x}{L}\bigg)\bigg|\leq &\sum_{j=1}^{\infty} \bigg[\dfrac{2\ell^2}{3}\bigg|\dfrac{\theta^1_j}{\gamma_j}\bigg|+\zeta\bigg|\frac{\theta_j^0}{\gamma_j}\bigg|+|\theta_j^0|\bigg]\\\leq& \dfrac{2\ell^2}{3}\bigg[\sum_{j=1}^{\infty} (\theta^1_j)^2\bigg]^\frac{1}{2}\bigg[\sum_{j=1}^{\infty} \dfrac{1}{\gamma^2_j}\bigg]^\frac{1}{2}+\sum_{j=1}^{\infty} \bigg[\zeta\bigg|\frac{\theta_j^0}{\gamma_j}\bigg|+|\theta_j^0|\bigg]<\infty,
		\end{split}
	\end{equation*}
	due to the regularity of the initial data. We argue similarly for the series related to $w$.

\par\bigskip
{\bf Acknowledgements.}
A. Falocchi is partially supported by the INdAM - GNAMPA project 2023 ``Modelli matematici di EDP per fluidi e strutture e proprietà geometriche delle soluzioni di EDP'',  INdAM - GNAMPA  project 2024 “Problemi frazionari:  proprieta' quantitative ottimali, simmetria, regolarit\`a, and he is supported by the MUR (Italy) grant Dipartimento di Eccellenza 2023-2027, Dipartimento di Matematica, Politecnico di Milano.

J.T. Webster's research was partially funded by NSF-DMS 2307538. He wishes to thank UMBC for granting him a productive sabbatical, resulting in this research, in the Spring of 2024. Additionally, he wishes to thank Politecnico di Milano for hosting him during that time.

\section{Appendices}\label{appendix}

\subsection{Dissipative dynamical systems}\label{app1}
In this section we closely follow the notation and conventions from \cite{chueshov,chla}. 

Let $(S_t,H)$ be a dynamical system on a complete metric space $H$. We say that   $(S_t,H)$ is {\it ultimately dissipative} iff it
possesses a bounded absorbing set $\mathcal B$, which is to say, for any bounded set $B$, there is a time $t_B$ so that
$S_{t_B}(B)\subset \mathcal B$. A dynamical system is
\textit{asymptotically compact} if there exists a compact set $K$
which is uniformly attracting, i.e., for any bounded set $D\subset
H$ we have that $$\displaystyle~
\lim_{t\to+\infty}d_{{H}}\{S_t D|K\}=0$$ 
in the sense of
the Hausdorff semidistance. We say that
$(S_t,H)$ is \textit{asymptotically smooth} if for any bounded, forward
invariant $(t>0) $ set $D$ there exists a compact set $K \subset
\overline{D}$ which is uniformly attracting (as above). A \textit{global attractor} $A\subset H$ is a closed, bounded set in $%
H$ which is invariant (i.e. $S_tA={A}$ for all $%
t \in \mathbb R$) and uniformly attracting (as previously defined).

The following   {\it if and only if}  characterization of global
attractors  is well-known \cite{Babin-Vishik,chla}
\begin{theorem}\label{0}
Let $(S_t,H)$ be an ultimately dissipative dynamical system in a complete
metric space  $A$. Then $(S_t,H)$ possesses a compact
global attractor ${A} $ if and only if $(S_t,H)$ is
asymptotically smooth.
\end{theorem}
For non-gradient systems---such as the one in this paper---the above theorem is often used to obtain the existence of a compact global attractor.

A generalized fractal {\em exponential attractor} for the dynamics $(S_t,H)$ is a forward invariant, compact set $ A_{\text{exp}} \subset H$ in the phase space, with finite fractal dimension (possibly in a weaker topology), attracting bounded sets with uniform exponential rate.  When we refer to $A_{\text{exp}}$ as a {\em fractal exponential attractor}, we are  meaning that $A_{\text{exp}}\subset H$ has fractal dimension in $H$, rather than in some weaker space. 

Here we define {\em quasi-stability} as our primary tool in the long-time behavior analysis.  In a quasi-stable dynamical system the difference of two trajectories can be decomposed into a uniformly stable part and a compact part, with controlled scaling of powers. The theory of quasi-stable dynamical systems has been developed rather thoroughly in recent years \cite{chueshov,chla}. More general definitions of quasi-stable dynamical systems are included in \cite{chueshov}. For ease of exposition and application in our analysis we focus on a  narrower definition. 

Informally, we note that: 
\begin{itemize}
\item Obtaining the quasi-stability estimate on the global attractor $A$ implies additional smoothness and finite dimensionality $A$. This follows from the so called squeezing property and one of Ladyzhenskaya's theorems, see \cite[Theorems 7.3.2 and 7.3.3]{chla}.
\item Obtaining the quasi-stability estimate on an absorbing ball implies the existence of a global attactor, as well as an exponentially attracting set; uniform in time H\"{o}lder continuity (in some topology) yields finite dimensionality of this exponentially attracting set (in said topology). 
\end{itemize}

Let us proceed with a formal discussion of {\it  quasi-stability}.
\begin{assumption}\label{secondorder} Consider second order (in time) dynamics yielding the dynamical system $(S_t,H)$, where $H=X \times Z$ with $X,Z$ Banach, and $X$ compactly embeds into $Z$.  Further, suppose $y= (x,z) \in H$ with $S_ty =(x(t),x_t(t))$ where the function $x \in C(\mathbb R_+,X)\cap C^1(\mathbb R_+,Z)$. 
\end{assumption}
\noindent Assuming \ref{secondorder} we focus on the second order, hyperbolic-like evolution problems.
\begin{definition}\label{quasidef}
With Assumption \ref{secondorder} in force, suppose that the dynamics $(S_t,H)$ admit the following estimate for $y_1,y_2 \in B \subset H$:
\begin{equation}\label{specquasi}
||S_ty_1-S_ty_2||_H^2 \le e^{-\gamma t}||y_1-y_2||_H^2+C_q\sup_{\tau \in [0,t]} ||x_1-x_2||^2_{Z_*}, ~~\text{ for some }~~\gamma, C_q>0,
\end{equation} where $Z \subseteq Z_* \subset X$ and the last embedding is compact. Then we say that $(S_t,H)$ is {\em quasi-stable} on $B$.
\end{definition}

We now run through a handful of consequences of the type of quasi-stability described by Definition \ref{quasidef} for dynamical systems $(S_t,H)$ satisfying Assumption \ref{secondorder}.
\cite[Proposition 7.9.4]{chla}
\begin{theorem}[Asymptotic smoothness]\label{doy}
If a dynamical system $(S_t,H)$, satisfying Assumption \ref{secondorder}, is quasi-stable on every bounded, forward invariant set $ B \subset H$, then $(S_t,H)$ is asymptotically smooth. Thus, if in addition, $(S_t,H)$ is ultimately dissipative, then there exists a compact global attractor $ A \subset H$. 
\end{theorem}

In \cite[Theorem 7.9.6 and 7.9.8]{chla} the authors provide the following result concerning improved properties of the attractor $A$, when the quasi-stability estimate is shown {\em on} the attractor $A$.
\begin{theorem}[Dimensionality and smoothness]\label{dimsmooth}
If a dynamical system $(S_t,H)$, satisfying Assumption \ref{secondorder}, has a compact global attractor $ A \subset H$, and is quasi-stable on $A$, then $ A$ has finite fractal dimension in $H$, i.e., $\text{dim}_f^HA <+\infty$. Moreover, any full trajectory $\{(x(t),x_t(t))~:~t \in \mathbb R\} \subset A$ has the property that
$$x_t \in L_{\infty}(\mathbb R;X)\cap C(\mathbb R;Z);~~x_{tt} \in L_{\infty}(\mathbb R;Z),$$ with 
$$||x_t(t)||^2_X+||x_{tt}(t)||_Z^2 \le C,$$
where the constant $C$ above depends on the ``compactness constant" $C_q$ in \eqref{specquasi}.
\end{theorem}

We may thus combine Theorems \ref{doy} and \ref{dimsmooth}, to obtain the following corollary: 
\begin{corollary}[Quasi-stability on absorbing ball]\label{doy*}
If a dynamical system $(S_t,H)$, satisfying Assumption \ref{secondorder}, is quasi-stable on a bounded absorbing set $ B \subset H$, then $(S_t,H)$  has a compact global attractor $ A \subset H$, and $ A$ has finite fractal dimension in $H$, i.e., $\text{dim}_f^HA <+\infty$. Moreover, any full trajectory $\{(x(t),x_t(t))~:~t \in \mathbb R\} \subset A$ has the property that
$$x_t \in L_{\infty}(\mathbb R;X)\cap C(\mathbb R;Z);~~x_{tt} \in L_{\infty}(\mathbb R;Z),$$ with bound
$$||x_t(t)||^2_X+||x_{tt}(t)||_Z^2 \le C,$$
where the constant $C$ above depends on the ``compactness constant" $C_q$ in \eqref{specquasi}.
\end{corollary}

The following theorem relates generalized fractal exponential attractors to the quasi-stability estimate \cite[p. 388, Theorem 7.9.9]{chla}
\begin{theorem}\label{expattract*}
Let Assumption \ref{secondorder} be in force. Assume that the dynamical system generated by solutions $(S_t,H)$ is ultimately dissipative and quasi-stable on a bounded absorbing set $ B$. We also assume there exists a space $\widetilde H \supset H$ so that $t \mapsto S_ty$ is H\"{o}lder continuous in $\widetilde H$ for every $y \in  B$; this is to say there exists $0<\alpha \le 1$ and $C_{ B,T>0}$ so that 
\begin{equation*}\label{holder}||S_ty-S_sy||_{\widetilde H} \le C_{ B,T}|t-s|^{\alpha}, ~~t,s\in[0,T],~~y \in  B.\end{equation*} Then the dynamical system $(S_t,H)$ has a generalized fractal exponential attractor $A_{\text{exp}}$ whose dimension is finite in the space $\widetilde H$, i.e., $\text{dim}_f^{\widetilde H} A_\text{exp}<+\infty$. 
\end{theorem}

\subsection{Nonlinear difference calculations}\label{app2}
We provide a central lemma which is instrumental in decomposing the nonlinear trajectories, as required in the proof of quasi-stability in Section \ref{proof-teo-main}. The computations are quite lengthy, but the lemma is necessary to the main result.

\begin{lemma}\label{zcable}
	Let $z=u^\mathcal{I}-u^{\mathcal{II}}$, let $h(u)$ the cable nonlinearity be as in \eqref{xii}$_3$. Also assume $u^i\in C^0([s,t], (H^2\cap H_0^1)(I))\cap C^1([s,t],L^2(\Omega))$ for $i=\mathcal{I},\mathcal{II}$. Then there exist $\eps>0$ and $C(\eps,T)>0$ such that
	\begin{equation}\label{ts}
		\bigg|\int_s^t\,_{-1}\langle h(u^\mathcal{I})-h(u^{\mathcal{II}}),z_{xt}\rangle_{1} d\tau\bigg|\leq \eps\int_s^tE_z(\tau)d\tau+C\sup_{\tau\in[s,t]}\|z(\tau)\|_{2-\eta}^2,\quad \forall\eta\in(0,\tfrac{1}{2}),
	\end{equation}
where
\begin{equation*}
	E_z(t):=\dfrac{\|z_t\|_0^2}{2}+\dfrac{\|z\|^2_{1}}{2}+\dfrac{\|z\|^2_{2}}{2}.
\end{equation*}
\end{lemma}
\begin{proof}
		We introduce preliminary the following function and its derivatives
	$$A(q):=\frac{q}{\sqrt{1+q^2}}\qquad A^{(1)}(q)=\frac{1}{(1+q^2)^{3/2}}\qquad A^{(2)}(q)=-\frac{3q}{(1+q^2)^{5/2}}\qquad A^{(3)}(q)=3\frac{4q^2-1}{(1+q^2)^{7/2}},$$
observing that $|A(q)|,|A^{(1)}(q)|,|A^{(2)}(q)|\leq 1$ and $|A^{(3)}(q)|\leq 3$ for all $q\in\mathbb{R}$.

	 Hence, by Lagrange mean value theorem we have
	\begin{equation*}
	\mathcal{L}(u^{\mathcal{I}})-\mathcal{L}(u^{\mathcal{II}})=\int_I A(q_x)z_x\, dx \qquad \forall q_x\in (u_{x}^\mathcal{I}+s_x\,,\,u_{x}^{\mathcal{II}}+s_x),
	\end{equation*}
	see the definition of $\mathcal{L}(\cdot)$ in \eqref{xii}, and
	\begin{equation*}
		\frac{u_{x}^\mathcal{I}+s_x}{\sqrt{1+[u_{x}^\mathcal{I}+s_x]^2}}-	\frac{u_{x}^{\mathcal{II}}+s_x}{\sqrt{1+[u_{x}^{\mathcal{II}}+s_x]^2}}=A^{(1)}(q_{x})z_x \qquad \forall q_x\in (u_{x}^\mathcal{I}+s_x\,,\,u_{x}^{\mathcal{II}}+s_x).
	\end{equation*}

	In this way we write
	\begin{equation*}
		\begin{split}
			h(u^\mathcal{I})-h(u^{\mathcal{II}})=&	\big[b\big(\mathcal{L}_0-\mathcal{L}(u^\mathcal{I})\big)-c\xi_0\big] \frac{u_{x}^\mathcal{I}+s_x}{\sqrt{1+[u_{x}^\mathcal{I}+s_x]^2}}-	\big[b\big(\mathcal{L}_0-\mathcal{L}(u^{\mathcal{II}})\big)-c\xi_0\big] \frac{u_{x}^{\mathcal{II}}+s_x}{\sqrt{1+[u_{x}^{\mathcal{II}}+s_x]^2}}\\
			=&\big[b\big(\mathcal{L}_0-\mathcal{L}(u^\mathcal{I})\big)-c\,\xi_0\,\big] A^{(1)}(q_x)z_x-b \bigg(\int_I A(q_x)z_x\, dx\bigg)\frac{u_{x}^{\mathcal{II}}+s_x}{\sqrt{1+[u_{x}^{\mathcal{II}}+s_x]^2}},
		\end{split}
	\end{equation*}
obtaining
	\begin{equation}\label{stima00}
		\begin{split}
			\,_{-1}\langle		h(u^\mathcal{I})-h(u^{\mathcal{II}}),z_{xt}\rangle_{1}=&\,_{-1}\langle	\big[b\big(\mathcal{L}_0-\mathcal{L}(u^{\mathcal{I}})\big)-c\,\xi_0\,\big] A^{(1)}(q_x)z_x,z_{xt}\big\rangle_{1}+\\&-b\bigg(\int_I A(q_{x})z_x\, dx\bigg)\,_{-1}\langle A(u_{x}^{\mathcal{II}}+s_x),z_{xt}\rangle_{1},
		\end{split}
	\end{equation}
	for all $q_x\in (u_{x}^\mathcal{I}+s_x\,,\,u_{x}^{\mathcal{II}}+s_x)$.
	
	Let us write the term in first line of \eqref{stima00} as
	\begin{equation}\label{stima0}
		\begin{split}
			\,_{-1}\langle\big[&b\big(\mathcal{L}_0-\mathcal{L}(u^\mathcal{I})\big)-c\,\xi_0\,\big] A^{(1)}(q_x)z_x\,,\,z_{xt}\rangle_{1}=\\&\dfrac{1}{2}\dfrac{d}{dt}\big(\big[b\big(\mathcal{L}_0-\mathcal{L}(u^\mathcal{I})\big)-c\,\xi_0\,\big] A^{(1)}(q_x)z_x\,,\,z_{x}\big)_0+\\
			-&\dfrac{1}{2}\,_{-1}\langle\big[b\big(\mathcal{L}_0-\mathcal{L}(u^{\mathcal{I}})\big)-c\,\xi_0\,\big] A^{(2)}(q_x)z^2_x\,,\,q_{xt}\big\rangle_{1}\\
			+&\dfrac{1}{2}\bigg(\int_I A(u_{x}^{\mathcal{I}}+s_x)u_{xt}^{\mathcal{I}}dx\bigg)\big(A^{(1)}(q_x)\,,\,z_x^2\big)_0\qquad\forall q_x\in(u_{x}^\mathcal{I}+s_x\,,\,u_{x}^{\mathcal{II}}+s_x),
		\end{split}
	\end{equation}
	where we used the time derivative rule.
	We integrate by parts in space the third line term
	in \eqref{stima0}
	\begin{equation*}\label{stima2}
		\begin{split}
		-\,_{-1}\langle\big[b\big(\mathcal{L}_0-\mathcal{L}(u^{\mathcal{I}})\big)-c\,\xi_0\,\big] A^{(2)}(q_x)z^2_x\,,\,q_{xt}\big\rangle_{1}=
			-&c\big(\,\xi_{0x} A^{(2)}(q_x)z^2_x\,,\,q_{t}\big)_0+\\&\big(\big[b\big(\mathcal{L}_0-\mathcal{L}(u^{\mathcal{I}})\big)-c\,\xi_0\,\big] A^{(3)}(q_x)q_{xx}z^2_x\,,\,q_{t}\big)_0+\\&
			\big(\big[b\big(\mathcal{L}_0-\mathcal{L}(u^\mathcal{I})\big)-c\,\xi_0\,\big]A^{(2)}(q_x)2z_xz_{xx}\,,\,q_{t}\big)_0,
		\end{split}
	\end{equation*}
	for all $q\in(u^\mathcal{I}+s\,,\,u^{\mathcal{II}}+s)$.
	Therefore, we obtain the bound on the third line term
	in \eqref{stima0}
	\begin{equation}\label{stima3}
		\begin{split}
			\big|\,_{-1}\langle\big[b\big(\mathcal{L}_0-\mathcal{L}(u^\mathcal{I})\big)-c\,\xi_0\,\big] A^{(2)}(q_x)&z^2_x\,,\,q_{xt}\rangle_{1}\big|
			\leq c\|\xi_{0x}\|_{L^\infty(I)}\|A^{(2)}\|_{L^\infty(I)}\|z_x\|_{L^\infty(I)}\|z_x\|_0\|q_{t}\|_0+\\&\big(b\|u^\mathcal{I}_x\|_{L^1(I)}+c\|\xi_0\|_{L^\infty(I)}\big)\|A^{(3)}\|_{L^\infty(I)}\|z_x\|^2_{L^\infty(I)}\|q_{xx}\|_0\|q_{t}\|_0+\\&2\big(b\|u^\mathcal{I}_x\|_{L^1(I)}+c\|\xi_0\|_{L^\infty(I)}\big)\|A^{(2)}\|_{L^\infty(I)}\|z_x\|_{L^\infty(I)}\|z_{xx}\|_0\|q_{t}\|_0\\&
			\leq \eps \dfrac{\|z\|_2^2}{2}+\dfrac{C}{2}\bigg(\dfrac{1}{\eps}+1\bigg)\|z\|^2_{2-\eta}\qquad \eta\in(0,1/2),\,\,\eps>0,
		\end{split}
	\end{equation}
		for all $q\in(u^\mathcal{I}+s\,,\,u^{\mathcal{II}}+s)$, having the same regularity of $u^\mathcal{I}$, $u^\mathcal{II}$.
	
	We integrate by parts in space the first integral in the fourth line term in \eqref{stima0} and bounding we get
	\begin{equation}\label{stima5}
		\begin{split}
			&\bigg|\bigg(\int_I A(u_{x}^\mathcal{I}+s_x)u_{xt}^\mathcal{I}dx\bigg)\big(A^{(1)}(q_x)\,,\,z_x^2\big)_0\bigg|\\\leq&\bigg(\int_I |u_{xx}^\mathcal{I}+s_{xx}||u_{t}^\mathcal{I}|dx\bigg)\|A^{(1)}\|^2_{L^\infty(I)}\|z_x\|_0\|z_x\|_0\leq \eps\dfrac{\|z\|_1^2}{2}+\frac{C}{2\eps}\|z\|_1^2\qquad (\eps>0).
		\end{split}
	\end{equation}
	It remains to consider the last term in \eqref{stima00}; integrating by parts in space as before, we find the bound
	\begin{equation}\label{stima1}
		\begin{split}
			&\bigg|-b\bigg(\int_I A(q_{x})z_x\, dx\bigg)\,_{-1}\langle A(u_{x}^{\mathcal{II}}+s_x),z_{xt}\rangle_{1}\bigg|\\\leq& \|A\|_{L^\infty(I)}\bigg(\int_I|z_x|dx\bigg)\|A^{(1)}\|_{L^\infty(I)}\|u_{xx}^{\mathcal{II}}+s_{xx}\|_{0}\|z_t\|_0\leq \eps\dfrac{\|z_t\|_0^2}{2}+\frac{C}{2\eps}\|z\|_1^2 \qquad (\eps>0).
		\end{split}
	\end{equation}
	Finally, from \eqref{stima00} and \eqref{stima0} we have
	\begin{equation*}\label{stima06}
		\begin{split}
			\bigg|\int_s^t	\,_{-1}\langle		h(u^\mathcal{I})-h(u^\mathcal{II}),z_{xt}\big\rangle_{1} &d\tau\bigg|\leq\dfrac{1}{2}\bigg|\Big[\big(\big[b\big(\mathcal{L}_0-\mathcal{L}(u^\mathcal{I})\big)-c\,\xi_0\,\big] A^{(1)}(q_x)z_x\,,\,z_{x}\big)_0\Big]_s^t\bigg|\\
			&+\dfrac{1}{2}\int_s^t\Big|\,_{-1}\langle\big[b\big(\mathcal{L}_0-\mathcal{L}(u^\mathcal{I})\big)-c\,\xi_0\,\big] A^{(2)}(q_x)z^2_x\,,\,q_{xt}\rangle_{1}\Big|d\tau\\
			&+\dfrac{1}{2}\int_s^t\Big|\bigg(\int_I A(u_{x}^\mathcal{I}+s_x)u_{xt}^\mathcal{I}dx\bigg)\big(A^{(1)}(q_x)\,,\,z_x^2\big)_0\Big| d\tau\\&
			+	b\int_s^t\bigg|\bigg(\int_I A(q_x)z_x\, dx\bigg)\,_{-1}\langle A(u_{x}^\mathcal{II}+s_x),z_{xt}\rangle_{1}\bigg|d\tau.
		\end{split}
	\end{equation*}
	Since
	$$
	\bigg|\Big[\big(\big[b\big(\mathcal{L}_0-\mathcal{L}(u^\mathcal{I})\big)-c\,\xi_0\,\big] A^{(1)}(q_x)z_x\,,\,z_{x}\big)_0\Big]_s^t\bigg|\leq C\sup_{\tau\in[s,t]}\|z(\tau)\|_{1}^2,
	$$
	collecting the inequalities \eqref{stima3},  \eqref{stima5} and \eqref{stima1} we obtain for all $\eta\in(0,1/2)$ and $\eps>0$
	\begin{equation*}\label{stima6}
		\begin{split}
			\bigg|\int_s^t\!	\,_{-1}\langle		h(u^\mathcal{I})-h(u^\mathcal{II}),z_{xt}\rangle_{1} d\tau\bigg|&\leq\! \eps\int_s^t\bigg(\dfrac{\|z_t(\tau)\|_0^2}{2}+\dfrac{\|z(\tau)\|_1^2}{2}+\dfrac{\|z(\tau)\|_2^2}{2}\bigg)d\tau+C(\eps,T)\!\sup_{\tau\in[s,t]}\!\|z(\tau)\|_{2-\eta}^2,
		\end{split}
	\end{equation*}
	i.e., we have obtained the inequality \eqref{ts}.
\end{proof}

\end{document}